\documentclass[12pt]{amsart}
\usepackage{graphicx}
\usepackage{xy} \xyoption{all}
\usepackage[utf8]{inputenc}
\usepackage{pgf,tikz,pgfplots}
\usepackage{mathrsfs}
\usetikzlibrary{arrows}
\usepackage{mathrsfs}
\usetikzlibrary{arrows}

\usepackage{todonotes}
\usepackage{amssymb}
\usepackage{geometry,comment}
\usepackage{hyperref}
\newtheorem{lemma}{Lemma}[section]
\newtheorem{corollary}[lemma]{Corollary}
\newtheorem{theorem}[lemma]{Theorem}
\newtheorem{proposition}[lemma]{Proposition}

\theoremstyle{definition}
\newtheorem{definition}{Definition}[section]

\newtheorem{remarks}[definition]{Remarks}
\newtheorem{remark}[definition]{Remark}

\newcommand{\Q}{\mathbb{Q}}
\newcommand{\R}{\mathbb{R}}

\newcommand{\N}{\mathbb{N}}
\newcommand{\Hh}{\mathbb{H}}

\DeclareMathOperator{\End}{End}
\DeclareMathOperator{\ann}{ann}

\def\a{\alpha}
\def\b{\beta}
\def\f{\phi}

\def\p{\pi}
\def\l{\lambda}
\def\g{\gamma} 
\def\d{\delta}
\def\m{\mu}
\def\n{\nu}
\def\t{\theta}
\def\ta{\tau}
\begin{document}
\title{ Leavitt path algebras and their representations}
\author[P. N. \'Anh]{Pham Ngoc \'Anh\textsuperscript{1\dag*}}
\keywords{free groups, free algebras, Fox derivatives, irreducibile polynomials, atomic factorizations, Leavitt algebras, localizations; word\\ \protect \indent * Corresponding author: P. N. \'Anh}
\thanks{This research was partially supported by Hungarian National Research Development and Innovation Office NKHIH K138828 and  K132951, by both VIASM and  Vietnamese Institute of Mathematics.}

\maketitle
\newcommand{\err}[1]{{\color{red}#1}}
\begin{center}
{\small 

\textsuperscript{1}HUN-REN Alfr\'ed R\'enyi Institute of Mathematics,\\ Re\'altanoda u. 13-15, Budapest, H-1053 Hungary\\
\textsuperscript{\dag}\url{anh@renyi.hu}

}
\end{center}

{\centering\footnotesize \it{Dedicated to paternal Professor Rich\'ard Wiegandt on the occasion of 
his $90^{th}$ birthday} \par}
\begin{abstract}
Viewing Leavitt path algebras of finite digraphs as rings of quotients defined by the ideal topology of the ideal generated by all arrows and sinks allows us to induce their representations from those of the quiver algebras and therefore provides a way to construct representations of Leavitt path algebras of not necessarily finite digraphs together with a computation of the endomorphism rings. This approach emphasizes the decisive role of infinite emitters, i.e., vertices with infinitely many outgoing arrows, in the representation theory of Leavitt path algebras. In particular, extensions of representations of ordinary quivers by taking tensor products are no longer simple extensions when infinite emitters exist, hence they become the targets of further study. Our results connect Leavitt path algebras to other vigorously active working areas in ring theory like localizations, quiver algebras, free associative algebras as well as noncommutative noetherian domains. Moreover, as Cuntz algebras ${\mathcal O}_n$
for operator graph algebras, our treatment emphasizes the central role of classical Leavitt algebras $L_K(1, n)$ in the study of Leavitt path algebras.
\end{abstract}
\maketitle
\section{Introduction}
\label{int} Although Leavitt algebras were already invented in the early 50's of the last century, the theory of Leavitt path algebras was substantially bolstered by the subsequent development of Cuntz algebras ${\mathcal O}_n$, its later twin. The aim of this work is a continuation of the investigation of simple modules over Leavitt path algebras initiated in \cite{an}. In contrast to \cite{an}, where ideas from $C^*$-algebras are used to treat certain irreducible representations, called {\em special simple} representatons, in this work we
use purely algebraic techniques to study their representation theory. These methods are  appropriate not only in ring theory but hopefully also in operator algebras. In any case,
our results suggest connections between Leavitt path algebras and different areas such as division algebras, noncommutative noetherian domains, quiver algebras and localizations. Hence one can not only borrow results, methods and ideas from operator algebras but also from the above mentioned branches to the study of Leavitt path algebras. Finitely presented modules over are also investigated in \cite{ara1} and \cite{ara2} by Ara and Brustenga by using localization technique, however, without relating to rings of quotients forming by Gabriel topology.

In Section \ref{c1} inspired by the Hilbert Nullstellensatz together with the fact that Leavitt path algebras of finite digraphs are perfect localizations of the ordinary quiver algebras, we use division algebra factors of quiver algebras, to construct new irreducible representations of Leavitt path algebras with pre-described endomorphism rings. Therefore one can apply a module theory of both quiver algebras and their factor algebras to Leavitt path algebras. In particular, there is a nice connection between Leavitt (path) algebras, division algebras and free associative (quiver) algebras. Consequently, with a few exceptions, a classification of irreducible representations reduces to the rather difficult setting of free associative algebras, so a mostly unsatisfying mission. This observation clearly underlines the important task of finding particular interesting well-behaved simple modules with nice invariants. In the last section we remove the finiteness of digraphs to obtain further
classes of simple modules, pointing out some applications to classical Leavitt algebras which arise from the well-known fact that free associative algebras of finite rank contain free associative subalgebras of infinite ranks. This shows again an obvious but useful role of division algebras in the study of Leavitt (path) algebras. In contrast to the basic idea of \cite{an} which originated in the analysis of characterizing continuous function algebras via their spectra, in this paper we investigate Leavitt path algebras by using the algebraic fact that they are rings of quotients of quiver algebras in Utumi's sense. In this way, one obtains both a new way and a further explanation for irreducible representations defined by Chen words. This effort leads to one main goal of this paper - defining representation spaces of Leavitt path algebras, or equivalently, to finding  "good" $K$-bases, i.e., Schreier (or Gr\"obner?) bases of their representations. This is important because in contrast to Leavitt path algebras where paths $\a \b^*$ do not form a $K$-basis by the appearance of Cuntz-Krieger condition (CK2), one can provide, for particular representations, their representation spaces with well-behaved actions making for a more transparent and better understanding.  In any case, the combined study of well understood subalgebras, like quiver algebras, ultramatricial subalgebras or subalgebras generated by commuting (symmetric) idempotents associated to ordinary paths, may lead to effective ways to handle modules over Leavitt path algebras.

\section{Preliminaries: rings of quotients, digraphs and their algebras}

\emph{A word about terminology}. All fields are commutative. All algebras are vector spaces together with an associative but not necessarily unital ring structure such that ring multiplication is commutable with scalars from a field. It is worth emphasizing that a base field is not required to be embedded in algebras over it. All algebras have local units and all modules over them are unitary, i.e., for every set of finitely many elements there exists an idempotent  which acts as an identity on this set.
\emph{Ideals} and \emph{modules} are considered with respect to algebras, that is,
they are also at the same time vector spaces over a field. 
\emph{Finitely presented} modules are factors of finitely generated modules by finitely generated submodules.  For an algebra $A$,
$A(1-x),\, (1-x)A, \, (1-x)A(1-x)$  are, by definition in order, sets $\{r-rx \,|\, r\in A\},\, \{r-xr\, | \, r\in A\}, \{r-xr-rx+xrx\,|\, r\in A\}$, respectively, for $x\in A$. For further undefined notions for rings, localizations or for Leavitt path algebras we refer to monographs \cite{s1} and \cite{aas}, respectively. 

The notion of \emph{rings of quotients} is introduced independently by Findlay and Lambek \cite{fl} and Utumi \cite{u1} in the late fifties, last century. Utumi's work is definitely important for use in Leavitt path algebras by permitting rings without identity. Namely, a ring $Q$ (not necessarily with identity) is a \emph{ring of right (left) quotients} of a subring $A$ and in this case $A$ is called \emph{dense on the right (left)} in $Q$ if for any two elements $q_1, q_2 \in Q, q_1\neq 0$, there is $r\in A$ such that $q_1r\, (rq_1)\neq 0,\, q_2r\, (rq_2)\in A$. It is trivial that the left (right) annihilator of a dense subring on the right (left) is trivial, i.e., $0$, and dense right (left)  ideals are essential. The alternative, more general approach of localizations for rings of quotients in which original rings are only homomorphically mapped, by using hereditary torsion theory
was developed in the sixties by Gabriel, Lambek, Morita, etc. 

Particular flat ring extensions are crucial in our investigation. A ring $B$ is called
a \emph{flat epimorphic right (left) ring of quotients} of a ring $A$ if it  is a flat left (right) $A$-module via a ring epimorphism $\f:A\rightarrow B$ 
i.e., two ring homomorphisms from $B$ are equal if they coincide on $\f(A)$. In this work, we are interested only in the case where $A$ can be considered as a subring of  $B$, i.e., they are bimorphic rings and so one can realize $B$ as a flat ring of quotients of $A$ with respect to the Gabriel topology $\frak T$ defined by right (left) ideals $R\, (L)$ of $A$ satisfying
$RB=B\, (LB=B)$. Left ideals $L$ of $A$ satisfying $L=BL\cap A$ are called $\frak T$-\emph{saturated}.  We shall need the following important result which is an obvious consequence of \cite{s1} Corollary XI.3.8.
\begin{corollary}\label{rightid} Let $B$ be a flat bimorphic left ring of quotients of a subring $A$. Then every left ideal $L$ of $B$ satisfies $L=B(B\cap L)$. In particular, left ideals $L$ of $A$ with $BL=B$ are finitely generated and form a Gabriel topology of $A$ to which $B$ is a flat ring of quotients of $A$.
\end{corollary} 
Since terminology in graph theory is notoriously nonstandard, we  fix
necessary notations and concepts to avoid confusion. 
We then remind the reader of the definitions of both "quiver" and "Leavitt path algebra", and make some germane observations about these algebras which will be useful in the sequel.  

A \emph{directed graph}, shortly \emph{digraph} is a quadruple $E=(E^0, E^1, s, r)$ of a non-empty \emph{vertex} set $E^0$, an  \emph{arrow} set $E^1$, and 
\emph{source} and \emph{range functions} $s, r:E^1\rightarrow E^0$. A vertex
$v\in E^0$ is \emph{singular} if it is either a \emph{sink} or \emph{infinite emitter} 
 according to $|s^{-1}(v)|= 0 $ or $|s^{-1}(v)| =\infty$, respectively. 
$v$ is a \emph{regular vertex} if it is not singular. A vertex $v$  is a \emph{source} if $|r^{-1}(v)|=0$.
A \emph{path} $\a$ of \emph{length} $n=|\a|>0$ from a {\it source} $s(\a)=v_0$ (to a {\it range} $r(\a)=v_n$ if $n< \infty$) is a sequence $\a=\overset{v_0}{\bullet} \overset{a_1}{\longrightarrow} \overset{v_1}{\bullet}\longrightarrow \cdots \overset{v_{l-1}}{\bullet} \overset{a_l}{\longrightarrow} \overset{v_l}{\bullet}{\longrightarrow}\cdots (l\leq n$  if  $n\in \N$, or $l< \infty$ if $n=\infty)$,  written as a word $\a=a_1\cdots a_i\cdots$, of 
arrows $a_i\in E^1$ satisfying $s(a_{i+1})=v_i=r(a_i)\, (i< n)$. For  each $l=0, 1, 2, \dots$ a subpath $a_1\cdots a_l\, (v_0=s(a_1)$ in case $l=0$) is called the \emph{head} $h_{\a}(l)$ or the \emph{initial segment of length} $l$ of $\a$ while the  remainder $a_{l+1}\cdots a_{l+i}\cdots=b_1b_2\cdots$ with $b_i=a_{l+1}, i\in \N$ is called the \emph{tail} $t_{\a}(l)$ {\it of colength} $l$
of $\a$.
Every vertex $v\in E^0$ is, by convention, a path of length $0$ from $v$ to $v$. Paths like $h_{t_{\a}(i)}(j)$ are called \emph{subpaths} of $\a$. $\a$ is called a \emph{finite path}, or from now on, simply a \emph{path} if its length is finite. The set of all finite paths (of length $n \geq 0$) is denoted by $F(E) \, (F_n(E)$, respectively). A path $\a$ is a \emph{closed path based at} $v$ if $\a \in F(E)$ and $v=s(\a)=r(\a)$, and so closed paths uniquely determine  their bases. 

Infinite paths are also called \emph{Chen words} whose set is denoted by $E^{\infty}$. Most simple examples of Chen words are \emph{rational Chen words} $\d^{\infty}$ where $\d$ are closed paths of positive length,  i.e.,  infinite paths $\d \d \cdots \d\cdots$. A \emph{period} $\p$ is a closed path of the smallest length defining the rational Chen word $\d^{\infty}$. Two Chen words $\a$ and $\b$ are called \emph{tail-equivalent}, denoted as 
$\a\sim \b$,  if there are integers $n, \, m\geq 0$ with $t_{\a}(n)=t_{\b}(m)$. The relation $\sim$ is trivially an equivalence relation;  the equivalence class of $\a \in E^{\infty}$ is denoted by $[\a]$. A Chen word $\a$ is called \emph{irrational} if it has no rational tails.
These definitions point out a striking similarity to the real numbers. While rational Chen words are of finite character, determined uniquely by their periods, irrational Chen words cannot be characterized by any of their finite heads.
A class $[\a]\, (\a\in E^{\infty})$ is called {\it rational} (resp., {\it irrational})  if it contains a rational (resp., irrational) Chen word. Chen classes are important because they provide natural, well-behaved bases for certain irreducible representations of Leavitt path algebbras [see \cite{an}].

A useful device in applications of digraphs is the composition of paths. Namely, if $\a=a_1\cdots a_m$ is a finite path and $\b=b_1\cdots $ is a not necessarily finite path of length $n, m\geq 0$, respectively, satisfying $r(\a)=s(\b)$, then the \emph{composition} or the \emph{product} $\a\b$ is well-defined by concatenation
$$\a\b=a_1\cdots a_mb_1\cdots b_i\cdots$$
The \emph{addition law of length} $|\a\b|=|\a|+|\b|$ holds obviously when $\a\b$ is well-defined.

A \emph{partial ordering} $\leq$ on $F(E)$ is defined by
$\a\leq \b$ if $\a$ is an initial segment of $\b$, i.e., $\b=\a \g$ for an appropriate $\g \in F(E)$.

Reversing arrows gives the \emph{dual digraph} $E^*$ of $E$, i.e., $E^*$ has the same vertex set $E^0$ and the set $\{a^* | a \in E^1\}$ of arrows with $s(a^*)=r(a),\, r(a^*)=s(a)$. Arrows, paths of positive length in $E$ or $E^*$ are called \emph{real} or \emph{ghost}, respectively. Paths in $E^*$ are exactly $\a^*=a_n^*\cdots a_1^*$ for 
$\a=a_1\cdots a_n\in F(E)$. 
Therefore the set of finite paths (of length $n$) in $E^*$ is denoted by $F^*(E)\, (F^*_n(E))$ not by $F(E^*)\, (F_n(E^*)$, respectively),  as is expected by convention, emphasizing the fact that $^*$ is an anti-isomorphism between $E$ and $E^*$.  

\begin{definition}
\label{q1} The \emph{path algebra} or the \emph{quiver algebra} $KG$ of a digraph $G$ over a field $K$ is the $K$-vector space $KG$ with base $F(G)$ such that a product 
$\m\n\, (\m, \n \in F(G))$ is the composition of paths when $r(\m)=s(\n)$, or $0\in KG$ when $r(\m)\neq s(\n)$.
\end{definition}

\begin{definition}
\label{q2} Let $\hat E$ be the extension of a digraph $E=(E^0, E^1, s, r: E^1\rightarrow E^0)$ by adding ghost arrows $\overset{r(a)}{\bullet} \overset{a^*}{\longrightarrow} \overset{s(a)}{\bullet}$ for all $a\in E^1$. The \emph{Leavitt path algebra} $L_K(E)$ of a digraph $E$ over a field $K$ is the factor of $K\hat E$ by the 
so-called \emph{Cuntz-Krieger relations}
\begin{enumerate}
\item \,\,\,\,  (CK1)\,\,\,\, for any two arrows $a, \, b\in E^1$\,\,\,\,\,\,\,\,\,\,\,\,\,\,\,\,\,\,\,  
 $a^*b=\begin{cases} r(a)& \text{if}\qquad a=b,\\ 0& \text{if}\qquad a\neq b\end{cases},$ 

\item \,\,\,\, (CK2)\,\,\,\, for every regular vertex $v\in E^0_r$ \,\,\,\,\,\,\,\,\,\,\,\,\,\,\,\,\,\,\,\,\,\,\,\,\,\,\,\,\,\,\, $v=\sum\limits_{a\in s^{-1}(v)}aa^*$. 
\end{enumerate}
The \emph{Cohn path algebra} $C_K(E)$ of $E$ is the factor of $K\hat E$ by (CK1).
\end{definition}
Although $KE$ and $KE^*$ are clearly subalgebras of both $C_K(E)$ and $L_K(E)$, the set  $\{\a\b^*\,|\, r(\a)=r(\b);\, \a,\, \b\in F(E)\}$ is a $K$-basis of $C_K(E)$ but only a generating set of the $K$-space $L_K(E)$ by the Cuntz-Krieger condition (CK2). This fact suggests a crucial problem of finding good $K$-bases for certain (irreducible) representations of $L_K(E)$ as was done in the case of special irreducible representations, see for example \cite{an}. Moreover, there is, fortunately, no confusion when 
$\sum k_i\a_i\b^*_i\in K\hat E$ are used simply for elements of both $C_K(E)$ and $L_K(E)$. By Cuntz-Krieger condition (CK1) the assignment 
$\a \in F(E)\mapsto \a^*\in F^*(E)$ induces a standard, canonical involution in both $L_K(E)$ and $C_K(E)$ respectively. The following basic result, Theorem 4.1 \cite{ajm} can be verified directly and nicely connects the module theory of Leavitt path algebras $L_K(E)$ to that of $KE$ as we will see in this work. For the sake of convenience, from now on $I$ will always denote the ideal of $KE^*\, (KE)$ generated by all ghost (real) arrows and sinks of $E$, respectively.  Moreover, the 
$I$-topology is denoted by $\frak T$. If $E$ is a finite digraph, then $\frak T$ is a  perfect right Gabriel topology of $KE$ or
 a left perfect Gabriel topology of $KE^*$, respectively. However, $\frak T$ is no longer a Gabriel topology by the absence of the notorious condition (T4) when $E$ has a finite vertex set with infinite emitter(s)!
\begin{theorem}[Theorem 4.1\cite{ajm}]\label{loc} Let $E$ be a digraph and $K$ a field. If $E$ is finite, then $L_K(E)$ is both a perfect left and right localization of $KE^*$ and $KE$ with respect to the Gabriel topology $\frak T$ defined by the two-sided ideal $I$ generated by sinks and ghost (real) arrows of $E$, respectively. If $E$ is infinite, then $L_K(E)$ is only a ring of both left and right quotients of $KE^*$ and $KE$, respectively.
\end{theorem}

\section{Irreducible representations of finite digraphs}
\label{c1}
In this section $E$ always denotes a finite digraph with the dual $E^*$; $K$ is a field and $I$ is an ideal of $KE^*$ generated by all sinks and ghost arrows of $E$. $\frak T$ denotes the Gabriel topology of $KE^*$ defined by powers $I^l \,(l\in \N)$.
Moreover, by definition, $L_K(1,n)\, (n\geq 1)$ is the classical Leavitt algebra over $K$ generated by $x_i, y_i(=x_i^*)$ subject to $\sum_{i=1}^{n} x_iy_i=1$ and $y_jx_i=\begin{cases} 1& \text{if $j=i$}\\ 0& \text{if $j\neq i$}\end{cases}$, that is, $L_K(1, n)$ is a Leavitt path algebra of one vertex with $n$ loops $x_i$, where $y_i$ are ghost loops associated to $x_i$. $K$ as the canonical factor of $K\langle y_1, \dots, y_n\rangle$ by sending all $y_i$ to $0$, is called \emph{trivial factor, module...} of $K\langle y_1, \dots, y_n\rangle$ when it is appropriate. We shall keep this notation throughout this work. First, we have immediately the following description of their simple modules by using Theorem \ref{loc}.
\begin{theorem}\label{simple1} Let $E$ be a finite digraph with the dual $E^*$ and $K$ a field.
If $\frak T$ is the ideal topology of $KE^*$ defined by the ideal generated by all ghost arrows and sinks in $E$, then there is a one-to-one correspondence between simple left $L_K(E)$-modules and left ideals
$L$ of $KE^*$ maximal in the set of left ideals of $KE^*$ which are not open. More generally, there is a bijection between cyclic left $L_K(E)$-modules and left $\frak T$-saturated ideals $L$ of $KE^*$, i.e.,  $L=L_K(E)L\cap KE^*$. In particular, a generator set $\{r_i\}\subseteq KE^*$ of $_{KE^*}L$ forms a relation set for $_{L_K(E)}M=L_K(E)/L_K(E)L\cong L_K(E)\otimes_{KE^*} (KE^*/L)$,
that is, $M$ is generated by a generator $m\in M$ subject to $r_im=0$ for each $r_i$.
\end{theorem}
\begin{remark} One can formulate Theorem \ref{simple1} in terms of $E$ if we use right $L_K(E)$-modules. However, experts in
Leavitt path algebras traditionally work with left modules. In view of this result the classification of simple modules of Leavitt path algebras is, in general, rather hard. For example, in the case of classical Leavitt algebra $L_K(1, n)\, (n>1)$ the classification problem reduces partially to one of all simple modules of free associative algebras of rank $n>1$, which is a rather difficult task. Furthermore, it is important even for digraphs for which the classification is not solvable, to find particular irreducible representations with good descriptions. In fact, the aim of this work is to provide new, well-behaved irreducible representations of not necessarily finite digraphs, and to connect Leavitt path algebras with other classical rings like Weyl algebras, simple domains and enveloping algebras of finite dimensional Lie algebras and even more generally to non-commutative noetherian domains, etc.
\end{remark}
\begin{proof} Cyclic (simple) left modules of $L_K(E)$ correspond bijectively to (maximal) left ideals of $L_K(E)$
and there is a one-to-one correspondence between (maximal) left ideals of $L_K(E)$. Therefore  Theorem \ref{simple1} is an obvious consequence of Corollary \ref{rightid} and Theorem \ref{loc}.
\end{proof}
We do not here discuss in its full generality the isomorphism problem of cyclic modules and the even more interesting and harder problem of determining their annihilator ideals. However, we shall describe isomorphism classes of certain simple modules. As a first invariant we compute  endomorphism rings of certain irreducible representations of Leavitt path algebras.
\begin{corollary}\label{endoring} Let $M=L_K(E)/L_K(E)L$ be an arbitrary cyclic left $L_K(E)$-module  where $L$ is a $\frak T$-saturated left ideal of $KE^*$. Then $\End(_{KE^*}KE^*/L)$ naturally becomes a subring of $\End(_{L_K(E)}M)$. Moreover, if
$L$ is a  maximal left ideal of $KE^*$ not open in $\frak T$, then $\End(_{KE^*}KE^*/L)$ is canonically isomorphic to the endomorphism ring of $_{L_K(E)}M$.
\end{corollary}
\begin{proof} The first claim is obvious in view of the isomorphism $M\cong L_K(E)\otimes_{KE^*} (KE^*/L)$ and induced endomorphisms $1\otimes \t$ of $M$ where $1$ is the identity on $L_K(E)$ and $\t$ is an endomorphism of $_{KE^*}KE^*/L$. For the last claim, let $\f$ be an arbitrary non-zero endomorphism of $M$. $\f$ is uniquely determined by the image $\sum_i\m_i\n^*_i+L_K(E)L$ of $1+L_K(E)L$ where $\m_i, \, \n_i$ are appropriate paths in $F(E)$. Then by a standard argument there is an appropriate path $\m\in F(E)$ such that $\m^*(\sum_i\m_i\n^*_i)\in KE^*\setminus L$ and the image $\f(\m^*+L_K(E)L)=\m^*\f(1+L_K(E)L)=\m^*(\sum_i\m_i\n^*_i)+L_K(E)L \in KE^*+L_K(E)L$  of 
 $\m^*+L_K(E)L=\m^*(1+L_K(E)L)$ is not zero. Consequently $\f=1\otimes \t$ where $\t$ is an endomorphism of the simple $KE^*$-module $KE^*/L$ sending $\m^*+L$ to $\m^*(\sum_i\m_i\n^*_i)+L$, completing the proof. 
\end{proof}
\begin{remark}\label{refine} We shall refine Corollary \ref{endoring} in the next section when we discuss irreducible representations of Leavitt path algebras of not necessarily finite digraphs.
\end{remark}
\begin{theorem}\label{iso1} Let $L_1$ and $L_2$ be two maximal left ideals of $KE^*$ which are not open in the topology $\frak T$, i.e., all $_{KE^*}KE^*/L_i$ and $_{L_K(E)}M_i=L_K(E)/L_K(E)L\, (i=1, 2)$ are simple. Then $_{L_K(E)}M_1\cong _{L_K(E)}M_2$ if and only if $_{KE^*}KE^*/L_1\cong _{KE^*}KE^*/L_2$, that is, for such modules, the isomorphism problem reduces to one over $KE^*$. 
\end{theorem}
\begin{proof} Let $\f:M_1\longrightarrow M_2$ be an arbitrary module isomorphism over $L_K(E)$. $\f$ is uniquely determined by the image $r+L_2$ of $1+L_1$. It is obvious by routine arguments that there is a path $\m\in F(E)$ such that $\f(\m^*+L_1)=\m^*f(1+L_1)=\m^*r+L\neq 0$ satisfying $\m^*r\in KE^*$. In view of the fact $\m^*\notin L_1$ this implies that $_{KE^*}KE^*/L_1\cong _{KE^*}KE^*/L_2$ under the module homomorphism sending $\m^*+L_1$ to $\m^*r+L_2$. The sufficiency is trivial. 
\end{proof}
Although finding maximal $\frak T$-saturated left ideals of $KE^*$ is, in general, not an easy task,
there are certain obvious maximal left ideals of $KE^*$ which are not open in the topology $\frak T$ whence they induce new irreducible representations of $L_K(E)$ which are not belonging to the set of special irreducible representations defined in \cite{an}. Namely, in the rest of this section we point out a way to construct simple modules of $L_K(E)$ by using simple modules over $KE^*$ defined by maximal left ideals of $KE^*$ which are not open with respect to the topology $\frak T$. One can list these simple modules, or equivalently maximal but not open left ideals of $KE^*$ in one theorem, however, for a particular use, it is better if we separate them. The first immediate but transparent result is the following. 
\begin{proposition}\label{newsim1} If $L$ is a maximal left ideal of $KE^*$ of infinite $K$-codimension, then $L$ is not open in the topology $\frak T$ whence $L_K(E)\otimes_{KE^*} KE^*/L$ is an irreducible
representation of $L_K(E)$.
\end{proposition}
\begin{proof} Since $KE^*/I$ is a finite product of copies of $K$ indexed by regular vertices, all the powers $I^l\, (l\geq 1)$ have finite $K$-dimension. If $L$ is open, then it contains some power $I^l$, therefore $L$  has a finite $K$-codimension, a contradiction.  This completes the proof.
\end{proof}   
In view of Proposition \ref{newsim1} we first focus our attention on maximal finite codimensional left ideals of $KE^*$ which are not open in $\frak T$. From the proof of Proposition 4.4 \cite{ajm} about the existence of strong Schreier basis one sees by the standard Lewin's argument \cite{le1} that finite codimensional left ideals of $KE^*$ are finitely generated. 
Therefore cyclic modules over $L_K(E)$ induces by finite codimensional left ideals of $KE^*$ are finitely presented of finite length. 
\begin{theorem}\label{fpl} There is a one-to-one correspondence between finitely presented modules $M=L_K(1,n)m$ of finte length over $L_K(1, n)$ and finite codimensional left ideals $L=\ann_{K\langle y_1, \dots, y_n\rangle}m$ of $K\langle y_1, \dots, y_n\rangle$ which are not open in $\frak T$. In particular, the length $lg(M)$ is a difference of the length $lg(K\langle y_1, \dots, y_n\rangle)/L$ and the number of trivial subfactors of $K\langle y_1, \dots, y_n\rangle/L$.
\end{theorem}
\begin{proof} The case $n=1$ is trivial. Therefore one can assume $n\geq 2$. First we consider the case when $M$ is simple. Hence $L_K(1,n)L$ is a maximal left ideal of $L_K(1,n)$, consequently $L$ is a left essential ideal of $K\langle y_1, \dots, y_n\rangle$ in view of the fact that $L_K(1, n)$ does not contain minimal left ideals. Since $M$ is finitely presented, $L$ is a finitely generated left ideal of $K\langle y_1, \dots, y_n\rangle$ by Corollary \ref{rightid}. Therefore $L$ is finite codimensional by (3.3) Theorem \cite{rr}. This implies that the maximal ideal of $K\langle y_1, \dots, y_n\rangle$ contained in $L$ is also finite codimensional whence this ideal is a finite product of maximal finite codimensional ideals and at least one of them is not $I$ by the fact that $L$ is not open. This implies that there is a maximal left ideal of $K\langle y_1, \dots, y_n\rangle$ containing $L$ which is not open. However, this claim contradicts to Corollary \ref{rightid}. Henceforth, $L$ is a maximal not open left ideal of $K\langle y_1, \dots, y_n\rangle$. The case when $M$ is finitely presented of finite length follows now immediately from induction and the obseration that $L_K(1,n)$ is hereditary and hence coherent whence finitely generated of finitely presented modules are again finitely presented.
\end{proof}
It is worth to note that although all maximal infinite codimesional left ideals $L$ are $\frak T$-saturated and then induce irreducible representations of $L_K(1,n)$ there are infinitely generated not maximal left ideals $L$ of $K\langle y_1, \dots, y_n\rangle$ which produce simple modules over $L_K(1,n)$. For non-trivial examples, just look at representations defined by
irrational Chen words. Moreover, the proof of Theorem \ref{fpl} provides the following strange property of the maximal flat ring of quotients, i.e. the fc-localization of free algebras of rank $\geq 2$ defined in \cite{rr}.
\begin{corollary}\label{nofp} The maximal flat ring of quotients, i.e., the fc-localization of the free algebra of rank $\geq 2$ has no simple finitely presented representations.
\end{corollary} 
The argument of Theorem \ref{fpl} implies also the following important result.
\begin{corollary}\label{fpfd} Finitely presented modules with a generator $m$ over $L_K(E)$ of a finite digraph $E$ is of finite length if and only if $\ann_{KE^*}m$ is finite codimensional.
\end{corollary}
\begin{proof} Since all finitely generated projective modules over $L_K(E)$ are cyclic, finitely presented modules are also cyclic. Therefore it is not a restriction that finitely presented modules have a generator $m$. Finite codimensional left ideals of $KE^*$ is finitely generated by the existence of strong Schreier bases for quiver algebras, see Proposition 4.4 \cite{ajm}.
This shows the sufficience. For the necessity it is enough to verify the case of finitely presented simple modules because $L_K(E)$ is a coherent ring. By Corollary \ref{rightid} $\ann_{KE^*}m$ is finitely generated. If $\ann_{KE^*}m$ is not an essential left ideal of $KE^*$, then the exactness of the functor $L_K(E)\otimes_{KE^*}-$ implies that $L_K(E)m$ is isomorphic to a minimal left ideal of $L_K(E)$ whence it can be generted by some sink $v$. This implies $\ann_{KE^*}m$ is of codimension 1. Therefore one can assume that $\ann_{KE^*}m$ is essential. In this case using proposition 4.4 \cite{ajm} one can argue in the same way as in the proof of (3.3) Theorem \cite{rr} to obtain that $\ann_{KE^*}m$ is finite codimensional.
\end{proof}
Theorem \ref{simple1} together with
Theorem \ref{iso1} offers a way to construct additional new irreducible representations of Leavitt path algebras which are not the special ones discussed in \cite{an}. Fortunately, Theorem \ref{cyclicalgebra} will show that every finite-dimensional division algebra over a field $K$ can be realized as an endomorphism ring of appropriate irreducible representations of a suitable Leavitt algebra $L_K(1,n)$. Moreover, these examples emphasize also an obvious fact that a base field $K$ plays a vital role in the classification problem of irreducible representations, for example, there is only one four dimensional real division algebra, namely the quaternion skew field $\Hh$. Consequently, by Theorem \ref{cyclicalgebra} the real division algebra $\Hh$  can be realized as an endomorphism ring of simple modules over the real Leavitt algebra $L_{\R}(1,2)$. Because of its decisive role and for the sake of better understanding, we discuss the case of classical Leavitt algebras separately although one could insert this treatment in the general theory carried out in the next section when we deal with not necessarily finite digraphs.

\begin{theorem}\label{cyclicalgebra} Let $K$ be a field and $D$ a central finite-dimensional division $K$-algebra generated by elements of a non-zero $n$-tuple $\frak =(d_1, \dots, d_n)\in D^n$. Then $D$ can be realized as the endomorphism ring of a finitely presented simple module $S^D_{\frak r}=L_K(1,n)/L(1,n)L$ where $L$ is a kernel of a $K$-algebra homomorphism from $K\langle y_1, \dots, y_n\rangle$ onto $D$ sending $y_i$ to $d_i\in D$ and $L_K(1,n)$ is viewed as a left flat localization of $K\langle y_1, \dots, y_n\rangle \, (y_i=x_i^*)$ with respect to the Gabriel (ideal) topology $\frak T$ defined by the ideal $I$ generated by all $y_i$'s. 
\end{theorem}
\begin{proof} The assumption implies that some of the $d_i$-s are not zero. Since $L$ is a maximal two-sided ideal of $K\langle y_1, \dots, y_n\rangle$ defined by the homomorphism from $K\langle y_1, \dots, y_n\rangle$ to $D$ sending $y_i$ to $d_i$, $L$ does not contain all $y_i$ and is both a maximal left and right ideal of $K\langle y_1, \dots, y_n\rangle$. We claim that $L$ is not open in the topology $\frak T$. Otherwise $L$ could contain some powers $I^l \,(l>0)$
whence being maximal prime ideals, $L=I$, a contradiction. The theorem follows immediately now from Theorem \ref{simple1} and Corollary \ref{endoring} if we show that $S$ is finitely presented, i.e. $L$ is a finitely generated left ideal of $KE^*$. Fortunately, this claim is an obvious consequence of the Schreier-Lewin formula \cite{le1} on the rank of one-sided ideals of finite codimension over $K$. In particular, $D$ is also a finitely presented $K$-algebra.
\end{proof}
The finite dimensionality of $D_K$ is used only at one point in the proof to ensure that there is a surjective homomorphism from a free associative algebra of finite rank to $D$ such that not all veriables go to $0$. Therefore one can drop the finite dimensionality of $D_K$ and instead can make a formally much weaker assumption that a $K$-algebra $D$ is finitely generated. By the Hilbert Nullstellensatz this condition is equivalent to the finite $K$-dimension of $D$ if $D$ is commutative. Since modules are surely not isomorphic if their annihilator ideals do not coincide, one can use Theorem \ref{iso1} to more generally reformulate Theorem \ref{cyclicalgebra} for getting new types of pairwise non-isomorphic simple modules with endomorphism rings isomorphic to $D$ as follows, where a division ring $D$ is a $K$-algebra generated by at most $n$ elements.
\begin{theorem}\label{cycalg1} Let $D$ be a division ring containing $K$ in its center such that $D$ can be generated by $n$ not necessarily different elements as a $K$-algebra. For any nonzero ordered $n$-tuple $\frak r=(r_1, \dots, r_n)\in D^n$ generating $D$ let $L=L_{\frak r}$ be the kernel of the surjective $K$-algebra homomorphism from $K\langle y_1, \dots, y_n\rangle$ to $D$ sending $y_i$ to $r_i$. Then   $S_{\frak r}=L_K(1,n)/L_K(1,n)L$ is a simple left $L_K(1, n)$-module whose endomorphism ring is isomorphic to $D$. Moreover, for any two generating $n$-tuples ${\frak r}_1, \, {\frak r}_2$ of $D$  the associated modules $S_{{\frak r}_i}$ are isomorphic if and only if
$L_{{\frak r}_1}=L_{{\frak r}_2}$, i.e., ${\frak r}_1={\frak r}_2$. In particular there are uncountably many pairwise non-isomorphic simple modules over $L_{\R}(1,2)$ whose endomorphism ring is isomorphic to the skew field $\Hh$ of quaternions.  
\end{theorem}
\begin{proof}The assumption implies that $L$ is a two-sided ideal which does not coincide with $I$, the ideal generated by all $y_i$. Consequently, $L$ is not open in the Gabriel ideal topology generated by $I$ whence $S=L_K(1,n)\otimes_{KE^*} D\cong L_K(1, n)/L_K(1, n)L$ is an irreducible representation of $L_K(1,n)$. Now, we show that if $S_{{\frak r}_i}$ are isomorphic, then ${\frak r}_1={\frak r}_2$. By Theorem \ref{iso1} the associated simple $K\langle y_1, \dots, y_n\rangle$-modules $K\langle y_1, \dots, y_n\rangle/L_{{\frak r}_i}$ are isomorphic, too. Therefore their annihilator ideals $L_{{\frak r}_i}$ coincide. Consequently, being the image of the ordered $n$-tuple $(y_1+L, \dots, y_n+L)$ with
$L_{{\frak r}_1}=L=L_{{\frak r}_2}$, the $n$-tuples
 ${\frak r}_1$ and ${\frak r}_2$ are equal. Therefore to complete the proof of Theorem \ref{cycalg1} it is enough to observe the fact that any pair $q=(q_1, q_2)$ of different purely imaginary quaternions $q_1, q_2$ of length 1 generates the real quaternion algebra $\Hh$. 
\end{proof}
\begin{remark}\label{force1} Theorem \ref{cycalg1} provides evidence for applications of localizations. The Leavitt algebras $L_K(1, n)\, (n>1)$ are simple whence their simple modules are faithful. Fortunately, passing to corresponding (not necessarily simple) modules over free associative algebras one can use their annihilator ideals to obtain pairwise nonisomorphic simple modules with the same endomorphism division ring.
\end{remark}
In view of Theorem \ref{cycalg1} we call simple modules defined by a division algebra $D$ which is a $K$-algebra generated by at most $n$ elements, together with a non-zero generating $n$-tuple ${\frak r}=(r_1, \dots, r_n)\in D^n$ \emph{Hilbert modules} and denote them as $S^D_{\frak r}$ of $L_K(1, n)$. Namely, $S^D_{\frak r}$ is the simple $L_K(1, n)/L_K(1, n)L$ module where $L$ is a maximal left and right ideal of $K\langle y_1, \dots, y_n\rangle \, (y_i=x^*_i)$ which is the kernel of the  $K$-algebra homomorphism from $K\langle y_1, \dots, y_n\rangle$ onto $D$ sending $y_i$ to $r_i$. The requirement that $\frak r\neq 0\in D^n$ ensures that $L$ is not open in the topology $\frak T$. Furthermore it is important to emphasize that generator
systems of $L$ as a two-sided ideal are, in general, not generator systems of $L$ as a left module over $K\langle y_1, \dots, y_n\rangle$. It is also worth noting that although generator systems of the left ideal $L$ as a module over $K\langle y_1, \dots, y_n\rangle$ form a set of relations for $S^D_{\frak r}$, $S^D_{\frak r}$ may admit much simpler and nicer relations, for example, see Remark \ref{chen} and Remarks \ref{hilbert0} below.  Pointing out a close relation to  Hilbert's generalized Nullstellensatz [\cite{mat1} Theorem 5.1 or \cite{k1} Section 1-3] and emphasizing both its beauty and importance we state separately the particular case of commutative finite dimensional field extensions $F$ of $K$ which can be generated by $n$ not necessarily different elements. Of course,  these examples can be extended naturally to  arbitrary (not necessarily finite) digraphs having enough closed paths. We shall investigate these aspects in the more general setting of arbitrary digraphs in the next section.  
\begin{corollary}\label{hilbert} Let $F$ be a finite algebraic field extension of $K$ generated by elements of a nonzero $n$-tuple ${\frak r}=(r_1, \dots, r_n)\in F^n$. If $L$ is a kernel of a homomorphism from $K\langle y_i,\dots, y_m\rangle$ onto $F$ sending $y_i$ to $r_i$. Then $F$ is isomorphic to the endomorphism ring of the finitely presented irreducible representation $S^F_{\frak r}=L_K(1, n)/L_K(1, n)L$.
\end{corollary}  
\begin{remark}\label{chen}
Let ${\frak r}=(r_1, \dots, r_n)\in K^n$ be a nonzero point, i.e., not the origin of the affine space $K^n$. It is well-known and, in fact, easy to verify that elements $y_i-r_i$ generate the two-sided ideal kernel $L$ of a $K$-algebra homomorphism from $K\langle y_1, \dots, y_n\rangle$ onto $K$ sending $y_i$ to
$r_i\in K$. Let $f_{\frak r}=\sum_{r_i\neq 0} r_ix_i-1\in L_K(1,n)$. Then $S^K_{\frak r}$ can be generated by an element $z$ subject to $f_{\frak r}z=0$.
Namely, one can see that the relation $f_{\frak r}z=0$ is equivalent to $r_iz=y_iz$ for all indices $i$.    For $r_j=0$, one trivially has $(y_j-r_j)z=y_jz=-(y_jf_{\frak r})z=-y_j(f_{\frak r}z)=0$. Moreover, for each index $j$ with $r_j\neq 0$ the equality $y_jf_{\frak r}=r_j-y_j$ implies $0=-y_j(f_{\frak r}z)=-(y_jf_{\frak r})z=(y_j-r_j)z$. The converse assertion is obvious by the Cuntz-Krieger condition (CK2) because  $\sum\limits_{i=1}^{n} r_ix_iz=\sum\limits_{i=1}^{n} x_i(r_iz)=\sum\limits_{i=1}^n x_iy_iz=(\sum\limits_{i=1}^n x_iy_i)z=z$ implies $f_{\frak r}z=0$
\footnote{I learned this argument from F. Mantese in her private communication.}.
Consequently, $\ann_{L_K(1, n)}z$ contains the left ideal $\bar L$ of $K\langle y_1,\dots, y_n\rangle$ generated by $r_i- y_i \,(i=1, \dots, n)$ which is free on the set $\{r_i-y_i \,|\, i=1, \dots, n\}$ of codimension 1 whence $L=\bar L$ holds and $\{r_i-y_i \,|\, i=1, \dots, n\}$ is a basis of the left ideal $L$, too. Therefore our claim follows immediately from Corollary \ref{hilbert}. We deduce in this manner
that the following irreducible representations of classical Leavitt algebras $L_K(1,n)$ found by Mantese\footnote{In a private communication.} are special cases of Hilbert modules.
\end{remark}

\begin{theorem}[Mantese]\label{man} To any nonzero element
$\frak r\in K^n \,(n\geq 2)$ let $f_{\frak r}=-1+\sum r_ix_i \in K\langle x_1, \dots, x_n\rangle$. Then
$S^K_{\frak r}=L_K(1, n)/L_K(1, n)f_{\frak r}$ is an irreducible representation of $L_K(1, n)$ whose endomorphism ring is isomorphic to $K$.
\end{theorem}
We shall generalize this result to arbitrary digraphs in Section \ref{fdigraph}.
\begin{remarks}\label{hilbert0}
\begin{enumerate}
\item  For a polynomial $q\in K[x]$ of degree $l$, define the \emph{reciprocal polynomial} $q^*$ of $q$ by the equality $q(x)=x^lq^*(x^{-1})$. Let $q$ be an irreducible polynomial with 1 as the constant term and $\d \in F(E)$ a period based at vertex $v$ starting with an arrow $a\in s^{-1}(v)$. If $S^q_{\d}$ is an irreducible Rangaswamy representation of $L_K(E)$ defined by $q$ and $\d$, i.e., $S^q_{d}=L_K(E)/L_K(E)q(\d) \cong L_K(E)\otimes_{K[\d]} K[\d]/K[\d]q(\d)$ are isomorphic to $KE/KEq(\d)$ as modules over $KE$, then for every arrow $b\in s^{-1}(v)\setminus \{a\}$ one immediately has $b^* v=b^*q(\d) \in L_K(E)q(\d)$. This implies that $\sum_{b\in s^{-1}(v)\setminus \{a\}}L_K(E)b^*\subseteq L_K(E)q(\d)$. Consequently, $L_K(E)q(\d)$ contains the left ideal $L$ of $KE^*$ generated by
all $b^* \, (s(b)=v, \, b \neq a)$ and $q^*(\d^*)=({\d^*})^l+r_1({\d^*})^{l-1}+\cdots +r_l$ if  $q=r_lx^l+\cdots +r_1x+1$. It is noteworthy that $q(\d)$ is, in general, not equal to $(\d)^lq^*(\d^*)$ and therefore except for the case when all vertices of $\d$ are regular, $L_K(E)\otimes_{K[\d]} K[\d]/K[\d]q(\d)$ is not isomorphic to  $L_K(E)\otimes_{K[\d^*]} K[\d^*]/K[\d^*]q(\d^*)$ although $K[\d]/K[\d]q(\d)$ and $K[\d^*]/K[\d^*]q^*(\d^*)$ are isomorphic field extensions of $K$. For a case when some of vertices of $\d$ are not regular, i.e., an infinite emitter, the simple module $L_K(E)\otimes_{K[\d]} K[\d]/K[\d]q(\d)$ is only a homomorphic image of $L_K(E)\otimes_{K[\d]} K[\d^*]/K[\d^*]q^*(\d^*)$ which is generally not simple.
It is very important to underline a basic fact that except for a few obvious cases, Rangaswamy irreducible representations are not
Hilbert modules. For example, let $K=\Q,\, q=x^2+1\in \Q[x]$ and $\d=x_1x_2\in L_{\Q}(1, n) \,(n>1)$ together with the field extension $F=\Q[i]$ of $K$ by the root $i=\sqrt{-1}$ of $f$, then the Rangaswamy module $S^q_{\d}=L_{\Q}(1, n)/ L_{\Q}(1, n)q(x_1x_2)$ is not isomorphic to any Hilbert module $S^F_{\frak r}$ where $\frak r$ is an arbitrary non-zero ordered pair $(r_1, r_2,\dots, r_n)$ of elements from $\Q[i]$. We will return to this discussion later. 
\item Let $F$ be an algebraic field extension of $K$ generated by $\frak r=(r_1, \dots, r_n)\neq 0\in F^n$. Then by passing from $r_i\neq 0$ to some multiple $ t_ir_i \,(0\neq t_i\in K)$ if necessary, one can assume  without loss of generality that for $r_i\neq 0$ there are irreducible polynomials $-1+f_i\in K[x] \, (f_i\in K[x]x)$ satisfying $-1+f_i(r_i)=0$. The map $y_i\mapsto r_i\in F$ defines a maximal ideal $L$ of $KE^*$ containing $-1+f_i(y_i)=y_i^{l_i}f^*_i=y^{l_i}_if(x_1, \dots, x_n)$ where $f(x_1, \dots, x_n)=-1+\sum_{r_i\neq 0} f^*_i(x_i)=f$ and $l_i$ is the degree of $f_i$, and $f^*_i$ are reciprocal polynomials of $f_i$, respectively. It is important to underline
that $-1+f_i(y_i)\, (i=1, \dots, n)$ do not generate either the two-sided or the left ideal $L$ 
if $\dim {_KF}\geq 2$. However,
one can use  Schreier technique \cite{le1}, \cite{rr} to find a generator set of $_{K\langle y_1, \dots, y_n\rangle}L$ including $1-f_i(y_i)\, (i=1, 2, \dots, n)$. 
 We shall study this question in Section \ref{fdigraph} in the discussion of irreducible representations of infinite digraphs.
\item Since $K\langle y_1, \dots, y_n\rangle\, (n>1)$ is a free associative algebra, one can use Schreier bases [see Rosenmann and Rosset \cite{rr}] to find sets of free generators for an arbitrary left ideal $L$. Consequently, for (finitely generated) ideals of the commutative polynomial algebra $K[y_1, \dots, y_n]\, (n>1)$ one can view them as left ideals of free associative algebras and find their set of free generators and then reduce this set to generators over $K[y_1, \dots, y_n]$ by using commutativity relations. Since
ideals of commuative polynomial algebras are not projective (or equivalently free by the  Serre-Quillen-Suslin Theorem), they do not have a free set of generators, only a minimal set of generators. This explains why it is more difficult to find their Gr\"obner bases than their
Schreier bases when we view them as left (right) ideals of the corresponding free associative algebras.   
\end{enumerate}
\end{remarks}
It is now straightforward to reformulate the above results and constructions on irreducible representations of classical Leavitt algebras for the case of arbitrary finite digraphs $E$. First we fix some notation for the sake of convenience. Let $I$ be the two-sided ideal of $KE^*$ generated by all ghost arrows and sinks of $E$, that is, sources of $E^*$. Consequently, $KE^*/I$ is a finite product of copies  of $K$ indexed by regular vertices of $E$. Furthermore, any ring homomorphism from $KE^*$ to a unital K-algebra $B$ is uniquely determined by the images, i. e., the substitution of ghost arrows and vertices of $E$ by elements of $B$. However, in contrast to the case of free associative algebras, not every function $\phi: E^0\cup (E^1)^* \rightarrow B$ such that the corresponding image of $\phi$ generates the $K$-algebra $B$, defines a surjective algebra homomorphism from $KE^*$ onto $B$. There are obviously additional constraints needed on $\phi$ for the natural extension of $\phi$ to become a well-defined algebra homomorphism. Therefore we introduce the following convention. When we speak about \emph{a substitution of vertices and ghost arrows of $E^*$ by elements in $B$}, then we already assume that this substitution canonically induces a well-defined algebra homomorphism from $KE^*$ onto $B$.
Moreover, by Theorem \ref{loc} $L_K(E)$ is a left perfect localization of $KE^*$ with respect to the Gabriel (ideal) topology $\frak T$ defined by $I$. It is important to keep in mind that in this case both $L_K(E)$ and $KE^*$ are unital $K$-algebras. Let $D$ be a division $K$-algebra which is a factor of $KE^*$ by an ideal $L$ not containing $I$. This requirement is equivalent to the condition that the image of the set $E_g$ consisting of all ghost arrows and sinks, i.e., singular vertices of $E$ is not zero. In particular, any nonzero function
 $\frak r: E_g\rightarrow D$ satisfying $D=K\langle {\frak r}(E_g)\rangle$ defines a maximal ideal $L$ of $KE^*$ not containing $I$ which is the kernel of the $K$-algebra homomorphism from $KE^*$ to $D$ sending $y\in E^g$ to ${\frak r}(y)\in D$. Then $S^D_{\frak r}=L_K(E)\otimes_{KE^*} D\cong L_K(E)/L_K(E)L$ is an irreducible representation of  $L_K(E)$ with the endomorphism ring canonically isomorphic to $D$. We summarize these results in the following 
\begin{theorem}\label{divi1} Let $E$ be a finite digraph with the set $E_g$ of ghost arrows and sinks; $K$ a field together with a division $K$-algebra $D$ which is a factor of $KE^*$. If $\frak r: E_g\rightarrow D$ is an arbitrary nonzero function satisfying $K\langle {\frak r}(E_g)\rangle =D$, then $S_{\frak r}=L_K(E)\otimes_{KE^*} D\cong L_K(E)/L_K(E)L$ is an irreducible representation of $L_K(E)$ where $L$ is the kernel of the algebra homomorphism from $KE^*$ to $D$ sending $y\in E_g$ to $\frak r(y)\in D$. Two such irreducible representations defined by ${\frak r}_i\, (i=1, 2)$ are isomorphic if and only if ${\frak r}_1= {\frak r}_2$.
\end{theorem}
One can now combine Theorem \ref{divi1} with the construction of simple Hilbert modules $S^D_{\frak r}$ to extend considerably the class of Hilbert modules by using Proposition \ref{newsim1} as follows. With the notation and hypothesis of Theorem \ref{divi1} consider an arbitrary but
fixed simple left module $T$ defining a primitive $K$-domain $D$ generated by at most $E_g$ elements.
Hence $T$ is a faithful left $D$-module. Let $\frak r: E_g\rightarrow D$ be an arbitrary nonzero functions with $P=K\langle \frak r(E_g) \rangle$. Then $T$ becomes a simple left $L_K(E)$-module if $\sum_{\in E^0}v$ acts as an identity map and every $z\in E_g$ acts on $T$ by multiplying with $\frak r(z)$. If $D$ is a skew field, then we already know that an annihilator of any nonzero element of $T$ is a maximal left ideal of $KE^*$ which is not open in the topology $\frak T$. If $D$ is a domain which is not a skew field, then $T$ cannot have finite dimension over $K$. Namely, if $T$ is finite dimensional over $K$, then $D$ must be a matrix ring over a finite dimensional division $K$-algebra, whence a skew field, a contradiction. Therefore $T$ is infinite dimensional over $K$ when an annihilator of any non-zero element of $T$ is a maximal left ideal of $KE^*$ which is not closed in the topology $\frak T$.  Consequently $T_{\frak r}=L_K(E)\otimes_{KE^*} T$ is an irreducible representation of $L_K(E)$ and two such representations $T_{{\frak r}_i}\, (i=1, 2)$ are isomorphic if and only if ${\frak r}_1={\frak r}_2$.
These considerations are summarized  as follows.
\begin{theorem}\label{primitive1} Let $D$ be a primitive $K$-domain together with a faithful irreducible representation $T$. If a nonzero function $\frak r: E_g \rightarrow D$ together with the substitution $\frak r(z)$ for $z\in E_e$ and $1$ for $\sum_{v\in E^0}v$ makes $T$ an irreducible representation of $KE^*$, then   $T_{\frak r}=L_K(E)\otimes_{KE^*} T$ is a simple left $L_K(1, n)$-module whose endomorphism ring is canonically isomorphic to $\End(_DT)$. Moreover, two modules $T_{{\frak r}_i}\, (i=1, 2)$ are isomorphic if and only if ${\frak r}_1={\frak r}_2$
\end{theorem}
Among the most important sources for examples of primitive domains are Weyl algebras $A_n (n\in \N)$ over fields of characteristic $0$ and simple
one-sided principal domains... These examples emphasize the important role of base fields, and connect Leavitt path algebras to other intensely researched subjects like rings of differential operators or Cohn's theory of domains and skew fields.

As we now show, employing a similar argument, one can clearly remove the requirement that $D$ is a primitive $K$-domain in Theorem \ref{primitive1} to obtain another way to construct simple modules. In particular, all the
approaches to constructing irreducible representations above are incorporated in the next result. 
\begin{theorem}\label{primitive2} Let $P$ be a primitive $K$-algebra together with a simple faithful $P$-module $T$. If a nonzero function $\phi: E_g \rightarrow D$ together with the substitution $\phi(z)$ for $z\in E_e$ and $1$ for $\sum_{v\in E^0}v$ makes $T$ an irreducible representation of $KE^*$, then $T_{\phi}=L_K(E)\otimes_{KE^*} T$ is an irreducible representation of $L_K(E)$ whose endomorphism ring is canonically isomorphic to $\End(_PT)$. 
\end{theorem}
\begin{proof} To prove Theorem \ref{primitive2} it is enough to verify that the annihilator $\ann_{KE^*} m$ of any nonzero element $m\in T$ is not open with respect to the Gabriel topology $\frak T$. Namely, if
 $\ann_{KE^*} m$ is open with respect to $\frak T$, then it must contain some proper power $I^l$ of the ideal $I$ generated by all ghost arrows and sinks. Since $I^l$ is an ideal of $KE^*$, it must be contained in the annihilator of $\ann_{KE^*}T$ which is the kernel $\ker \phi$ of $\phi$ which contradicts the assumption. This completes the verification. 
\end{proof}

Till now we've dealt exclusively with maximal $\frak T$-saturated left ideals of $KE^*$ which are at the same time maximal left ideals of $KE^*$. Moreover, maximal non-open left ideals of $KE^*$ are constructed by using nice factor algebras of $KE^*$ like division rings, domains or primitive algebras together with suitable substitutions. In this way one can immediately see that the module theory of simple algebras, which are factors of $KE^*$ by ideals which are not equal to $I$, the ideal generated by all ghost arrows and sinks, is part of that for $L_K(E)$. Examples of both maximal $\frak T$-saturated left ideals of $KE^*$ which are not maximal left ideals of $KE^*$, and further maximal non-open left ideals of $KE^*$ are presented in the next section when we discuss
the case of not necessarily finite digraphs. 

\section{Simple modules of arbitrary digraphs}
\label{fdigraph}   

In this section we remove the finiteness condition from digraphs. From now on $E$ will refer to an arbitrary not necessarily finite digraph. $E^*$ is, as usual, the converse digraph of $E$ and $K$ is a commutative field. Moreover, $I$ always denotes the ideal of $KE^*$ generated by all ghost arrows and sinks of $E$. Consequently, $I$ is a finitely generated left ideal if and only if $E$ is finite. The ideal topology defined by $I$ is again denoted by $\frak T$. Then $\frak T$ is a flat Gabriel topology if and only if $E$ is finite. Moreover, $\frak T$ is Hausdorff if and only if there are no sinks, i.e., every vertex of $E$ emits at least one arrow. The aim of this section is to study special irreducible representations with a generator $m$ of $L_K(E)$ discussed in \cite{an}  in light of localizing techniques presented in the previous section, i.e., to determine $\ann_{KE^*}m$. It turns out that this comparison leads to interesting results and a better understanding of simple modules. 

Let $S$ be an arbitrary simple left $L_K(E)$-module and $m\in S$ an arbitrary non-zero element of $S$.
One can assume without loss of generality $m=vm$ for some vertex $v\in E^0$. By the Cunz-Krieger Condition (CK1) one has $am\neq 0$ for every arrow $a\in r^{-1}(v)$. Therefore the first obvious case appears when $S$ is required to satisfy $a^*m=0$ for all arrows $a\in s^{-1}(v)$. This strong additional condition seems to be very fortuitous. Namely, it can happen only when $v$ is either a sink, i.e., $s^{-1}(v)$ is empty, or $v$ is an infinite emitter, i.e. $s^{-1}(v)$ is infinite, as a consequence of the Cuntz-Krieger condition (CK2).  In the first case, $S$ is isomorphic to the minimal left ideal $L_K(E)v=(KE)v$. The second case leads to the definition of the Cohn-Jacobson simple module 
$S_v=\frac{L_K(E)v}{\sum\limits_{s(a)=v}L_K(E)a^*}$ introduced by Rangaswamy \cite{an} for an infinite emitter $v$. Excluding these cases, one can assume that there are infinite paths, i.e., Chen words $\a$ started on the left from $v$  satisfying $h_{\a}(l)^*m\neq 0$ for every $l\in \N$. This approach led to both Chen modules $S_{[\a]}$ defined by Chen words $\a$ which are denoted as $V_{[\a]}$ by Chen, and Rangaswamy modules $S^{f}_{\pi}$ defined by irreducible
polynomials $f$ with nonzero constant terms in one variable and periods $\pi$. We refer to \cite{an} for details. 

In the spirit of presenting a complete treatment, let $V$ be a vector space on the set of all Chen words, i.e., all infinite paths $\a$ on the right, together with an action of $L_K(E)$  induced by putting
$$\m^*\b= \begin{cases} t_{\b}(l) & \text{if}\qquad \m=h_{\b}(l),\\ 0& \text{if $\m$ is not a head of\, $\b$}\end{cases}\,\, $$  
for any finite path $\m$ and Chen word $\b$. Then $V$ becomes a semisimple $L_K(E)$-module which is a direct sum of simple modules $S_{[\a]}$ generated by $[\a]$ which runs over all tail-equivalent classes.

Consider first the simple Chen module  $S_{[\a]}$ defined by an irrational Chen word $\a=a_1a_2\cdots$ started on the left from $v\in E^0$. It is then quite easy to describe the left annihilator of $\a$ which is a maximal left ideal of $L_K(E)$. With this objective in mind, define the commuting idempotents $e^{\a}_l=a_1\cdots a_la^*_l\cdots a^*_1=h_{\a}(l)h^*_{\a}(l)=e_l$ simply for every $l\geq 0$. Then $e_0=v$ and $e_{l+n}e_l=e_{l+n}$ hold for any two $l, n \in \N$. Moreover one can see immediately (cf. Theorem 3.4 \cite{an}) ${\frak L}_\a=\ann_{L_K(E)}\a=L_K(E)(1-v)\oplus \bigoplus\limits_{l=0}^{\infty}L_K(E)(e_l-e_{l+1})$.
This implies that the left ideal $L^{\a}=\ann_{L_K(E)}\a \cap KE^*$ is generated by $F^*(E)\setminus \{h^*_{\a}(l)=h^*(l)\, |\, l\in \N\}$ and therefore $\{h^*(l)\,|\, l=0,1, 2, \dots\}$ is exactly the Schreier basis of $L^{\a}$ with respect to the standard basis $F^*(E)$ of $KE^*$. Consequently, $L^{\a}$ is not open in $\frak T$ because there are ghost paths of arbitrary length which do not belong to it. $L^{\a}$ is clearly not a maximal left ideal of $KE*$ because $L^{\a}_m=L^{\a}+\sum\limits_{l=m}^{\infty}Kh^*(l)$ form an infinite chain of left ideals with one-dimensional simple factors $L^{\a}_m/L^{\a}_{m+1}$.   
Put ${\bar T}_{\a}=KE^*/L^{\a}$. We claim that any non-zero submodule of ${\bar T}_{\a}$ contains almost all ${\bar L}^{\a}_m=L^{\a}_m/L^{\a}$, that is, almost all elements ${\bar h}^*(l)=h(l)+L^{\a}$. Namely, if $L$ is any left ideal of $KE^*$ containing $L^{\a}$ properly, then the set
$\{h^*(l)\,|\, l\in \N \}$ is not linearly independent $\mod L$ over $K$. Hence there is a smallest integer $n$ such that $h^*(n)\equiv \sum\limits_{i=1}^{m} k_ih^*(n_i) \mod L$ with possibly smallest $0< m\in N$ where $0\neq k_i\in K$ and $n<n_1<\dots <n_m$. Consequently, one has for all positive integer $p\in \N$
$$h^*(n+p)=a^*_{n+p}\cdots a^*_{n+1}h^*(n)\equiv a^*_{n+p}\cdots a^*_{n+1}(\sum\limits_{i=1}^{m} k_ih^*(n_i)) \mod L.$$
Since $\m^*h^*(l)\in L^{\a}$ for all $\m$ such that $h(l)\m$ is not a head of $\a${\textcolor{red}{,}} there are only two possible cases. There is either a smallest positive integer $p$ with $h^*(p+n)\in L$ or $h^*(n+p)\notin L$ for all positive integers $p$. In the first case, $L$ contains the submodule $L^{\a}_{p+n}$ for this smallest $p$ whence $L$ contains $I^{p+n}$ and $L$ is then open in 
$\frak T$.
In the second case there is some index $i\in \{1, \dots, m\}$ such that 
$h^*(n+p)$ and $a^*_{p+n}\cdots a^*_{n+1}h^*(n_i)$ are not contained in $L^{\a}$ for all positive integers $p\in \N$. Hence $t_{\a}(n)$ and $t_{\a}(n_i)$ are equal. This shows the equality $t_{\a}(n)=a_{n+1}\cdots a_{n_i}t_{\a}(n)=\d t_{\a}(n)$ whence by Proposition 2.1 \cite{an} $t_{\a}(n)$ is the rational Chen word ${\d}^{\infty}$ with $\d=a_{n+1}\cdots a_{n_i}$, a contradiction to the irrationality of $\a$. This completes a direct verification of the fact that $L^{\a}$ is maximal in the set of left ideals of $KE^*$ which are not open in $\frak T$. It is worth noting that $\frak T$ is a flat Gabriel topology if and only if $E$ is finite, and then by Theorem \ref{loc} and Corollary \ref{rightid} we already know that $L^{\a}$ is maximal in the set of all non-open left ideals of $KE^*$ and the equality $\ann_{L_K(E)}\a=L_K(E)L^{\a}$ holds. In this case it is also already known that $S_{[\a]}$ is isomorphic to $T_{\a}=L_K(E)\otimes_{KE^*}\frac {KE^*}{L^{\a}}=L_K(E)\otimes_{KE^*}{\bar T}_{\a}$. More generally, if all vertices $s(t_{\a}(l)) \, (l\geq 0)$ of $\a$ are regular, then successively applying the Cuntz-Krieger condition (CK2) implies that all elements $e_l\, (l\geq 0)$ are congruent to $\bar v=v+L^{\a}\in KE^*/L^{\a} \, \mod L_K(E)L^{\a}$, hence all $h(l)\otimes {\bar h}^*(l) (l\in \N)$ are equal to $v\otimes {\bar v} \, (l\in \N)$. Therefore the isomorphism $S_{[\a]}\cong T_{\a}$ holds also in this case. However, if some vertex of $\a$ is an infinite emitter, then elements $e_l (l\geq 0)$ are no longer congruent $\mod L_K(E)L^{\a}$ and therefore $S_{[\a]}$ is only a simple factor module of $T_{\a}$. For example, if $v_m=s(a_{m+1})$ is an infinite emitter, then $a^*_m\cdots a^*_1=h^*(m)$ does not belong to $L^{\a}_{m+1}$ whence the subfactor $((KE^*)h^*(l)+L^{\a}_{m+1})/L^{\a}_{m+1}$ of $KE^*/L^{\a}={\bar T}_{\a}$ is isomorphic to a simple Cohn-Jacobson module $S_{v_m}$ defined by $v_m$. 
In any case, the investigation of $T_{\a}$ seems to be an extravagant, delightful endeavor. Before discussing consequences,
one can summarize our results with respect to the notation described above as follows.
\begin{theorem}\label{irr} Let $\a$ be an irrational Chen word of $E$, i.e., an element of an irreducible representation $S_{[\a]}$ of $L_K(E)$ over a field $K$; $\frak T$ the ideal topology of $KE^*$ induced by the ideal $I$ generated by all ghost arrows and sinks of $E$.  Then $L^{\a}=\ann_{L_K(E)}\a \cap KE^*$ is not open but any left ideal $L$ of $KE^*$ strictly containing $L^{\a}$ is open in $\frak T$, namely, $L$ contains almost all
$h^*(l)=a^*_l\cdots a^*_1 (l\in \N)$. Moreover, $S_{[\a]}$ is only a homomorphic image of $T_{[\a]}= L_K(E)\otimes_{KE^*}\frac {KE^*}{L^{\a}}$ and an isomorphism appears when all vertices of $\a$ are regular. 
\end{theorem} 
\begin{remark}\label{infilength}The most obvious example for the last claim of Theorem \ref{irr} is provided by the \emph{rose $R_{\infty}$ with infinitely many petals} having one vertex $v$ and infinitely many loops $a_n\, (n\in \N)$. Then $\a=a_1a_2\cdots a_n\cdots $ is an irrational Chen word, and with the notation described before Theorem \ref{irr} one obtains a descending chain of left ideals
${\frak L}_n\, (n\in \N)$ of $L_K(R_{\infty})$ generated by all ghost arrows except $a^*_1, a^*_2a^*_1, \dots, a^*_{n-1}\cdots a^*_1$ where $n$ runs over all non-negative integers. Therefore if we denote by ${\frak L}_{\infty}$ the left ideal generated by all ghost arrows except all $h^*(l)\, (l\geq 0)$, then ${\frak L}_\a={\frak L}_{\infty}$ and a submodule of ${\frak L}_n/{\frak L}_{n+1}$ generated by $a^*_{n-1}a^*_{n-1}\cdots a^*_1+{\frak L}_{n+1}$ is isomorphic to the Cohn-Jacobson module $S_v$ defined by $v$ for each $n\in \N$. Therefore $T_{[\a]}$ contains an infinite chain of submodules such that the consequent subfactors are simple and isomorphic to $S_v$. Hence $T_{[\a]}$ has at least in this case two nonisomorphic simple factors $S_v$ and $S_{[\a]}$.
\end{remark} 
If $\frak I$ is any automorphism $KE^*$ fixing $I$, i.e., ${\frak I}(I)=I$, then $\frak I$ invariantly preserves the topology $\frak T$. Therefore we see the advantage in the case of finite digraphs when $L_K(E)$ is defined as the ring of quotients of $KE^*$ by
$\frak T$. Consequently, if we put $b^*={\frak I}(a^*)$ for each ghost arrow $a^*$, and let $b$ be the corresponding left $KE^*$-module homomorphism from $I$ to $KE^*$, then we obtain another involution of $L_K(E)$ and also another representation of elements of $L_K(E)$ as linear combinations $\sum k_i\m_i\n^*_i$ where $\m_i, \n_i$ are paths in the "real" arrows $b$. It is important to emphasize that $\n^*_i$ are elements of $KE^*$ but the $\m_i$ are no longer elements of $KE$ although they are elements of $L_K(E)$. In the most important case of the classical Leavitt algebra $L_K(1, n)$ the situation is more transparent. Namely, using the notation of Section \ref{c1}, let $z_1, \cdots, z_n$ be another basis of the left module $I=\sum K\langle y_1, \dots, y_n\rangle y_i$ over $K\langle y_1, \dots, y_n\rangle$. It is well-known that $I$ is a free non-unital associative algebra inside the free unital algera $K\langle y_1, \dots, y_n\rangle$. The equality $\sum K\langle y_1, \dots, y_n\rangle z_i=\oplus K\langle y_1, \dots, y_n\rangle z_i=I$ implies 
that $K\langle y_1, \dots, y_n\rangle$ is also a free algebra on the set $\{z_1, \dots, z_n\}$, i.e., $K\langle y_1, \dots, y_n\rangle=K\langle z_1, \dots, z_n\rangle$. Then a map $\frak I$ sending $y_i$ to $z_i$, respectively, is an automorphism of $K\langle y_1, \dots, y_n\rangle$ fixing $I$, and  define $s_i\colon I\rightarrow K\langle y_1, \dots, y_n\rangle$ sending $z_i$ to 1 and all the other $z_j, j\neq i$ to 0. By letting $z_i=s^*_i$ for each $i$, one obtains another involution of $L_K(1,n)$ with elements written as linear combinations $\sum k_{i...j...}s_{i_1}\cdots s_{i_l}z_{j_1}\cdots z_{j_m}$. Furthermore $\frak I$ extends obviously to an automorphism of $L_K(E)$. 
One can now form paths and Chen words in terms of new arrows $b$ and therefore one gets simple modules given by
irrational tail-equivalent classes $[\b]$ which we denote by $S^{\frak I}_{[\b]}$ and again call them Chen modules. One can get $S^{\frak I}_{[\b]}$ from $S_{[\a]}$ as follows. If we twist a module structure of $S_{[\a]}$ by putting $r\star m={\frak I}(r)m$ for any $r\in L_K(E)$ and $m\in S_{[\a]}$ then we obtain a new module, denoted as $_{\frak I}S_{[\a]}$. Moreover, the map $y_i\mapsto z_i \, (i=1, 2, \dots, n)$ also induces a bijection from irrational Chen words in the $x_i$ to ones in the $s_i (s^*_i=z_i)$, respectively. It is noteworthy to emphasize that the $s_i$ are not necessarily elements of $KE$. If we denote this map also by $\frak I$, and $\b$ is the corresponding image of $\a$, then $\frak I$ is clearly an isomomorphism between $S_{[\a]}$ and $_{\frak I} S^{\frak I}_{[\b]}$. It is well-known and easy to verify that for an automorphism $\phi$ of a ring $A$ together with a left $A$-module $M$, the twisted $A$-module $_{\phi}M$ by putting $r\star m=\phi(r)m \, (r\in A; m\in M)$ is isomorphic to $_AM$ if $\phi$ is an inner automorphism. However, this does not mean that they are not isomorphic if $\phi$ is not an inner automorphism. For example, if $\frak I$ is an automorphism of $L_K(1, 4)$ by fixing $y_1, y_2$ and interchanging $y_3$ with $y_4$. Let $\a$ be an arbitrary irrational Chen word in $y_1$ and $y_2$ then ${\frak I}(\a)=\a$ and one has $S_{[\a]}\cong {_{\frak I}S^{\frak I}_{[\a]}}$ because the intersection of the annihilators of both $\a \in S_{[\a]}$ and $\a=\b \in {_{\frak I}S^{\frak I}_{[\b]}}$ with $K\langle y_1, y_2, y_3, y_4\rangle$ are equal according to their description before Theorem \ref{irr}. Therefore it is interesting to search for reasonable conditions on the automorphisms $\frak I$ fixing $I$ such that the induced simple modules $_{\frak I}S_{[\a]}$ associated to irrational Chen words are not isomorphic to any old modules.

Now we turn to examine rational Chen words.
 Let $\a=\d^{\infty}=\d \d \cdots \, (\d=a_1a_2\cdots a_n$; $s(\a)=v)$ be a rational Chen word in $V$ defined by a period $\d$. Then $\ann_{L_K(E)} \a$ contains a left
ideal $L$ generated by vertices $\neq v=(\d^*)^0$; all ghost arrows $\neq a^*_1$, all ghost paths $\neq a^*_2a^*_1$ of length 2... all ghost paths $\neq a^*_{n-1}\cdots a^*_1$  of length $n-1=|\d|-1$, all paths $\neq  \d^*$ of length $|\d|$; all ghost paths $\neq a^*_1 \d^*$ of length $|\d|+1$ and so on. More precisely, $L$ is generated by all ghost paths $\neq a^*_m\cdots a^*_1(\d^*)^l$ of length $l|\d|+m (0\neq m <n-1)$ running over all non-negative integers. $L$ is obviously not open in $\frak T$ because it contains ghost paths of arbitrary length. Furthermore $L$ is also not a maximal left ideal of $KE^*$. This fact becomes transparent when we consider the Leavitt algebra $L_K(1, 2)$ and the rational Chen word $x^{\infty}_1$ where $L$ is exactly the two-sided ideal $\langle y_2\rangle\vartriangleleft K\langle y_1, y_2\rangle \, (y_i=x^*_i)$. It is fortuitous that the left ideal $L^q_{\d}$ generated by $L$ and a relation $q(\d^*)$ for any irreducible polynomial $q$ with nonzero constant term in one variable is a maximal left ideal of $KE^*$. Namely, consider an arbitrary element $r\in KE^*\setminus L^q_{\d}$. By the irreducibility of $q$ one can assume without loss of generality $r\notin K[\d^*]$.  Consequently, there are nonzero polynomials $q_1, \dots, q_l (l\leq n)$ of degrees $< \deg q$ and proper heads $h_1, \dots , h_l$ of $\d$ with $0\leq |h_1|<|h_2|<\dots<|h_l|<n$ with 
\begin{align}\label{eq1}
r\equiv \sum\limits_{i=1}^{l} h^*_iq_i(\d^*) \quad \mod L^q_{\d}.
\end{align}
As in the case of irrational Chen words one can multiply both side of \eqref{eq1} with heads of $t_{\d^{\infty}}$ on the left, and obtain a nonzero element of $K[\d^*] \mod K[\d^*]q(\d^*)$ whence $KE^*r+L^q_{\d}=KE^*$, or there is some index $j_m\neq j_1$ with $h_1\d^{\infty}=h_2\d^{\infty}$. Therefore $\d$ is not a period by Proposition 2.1 \cite{an} which is impossible. Therefore $L^q_{\d}$ is a maximal left ideal of $KE^*$ as claimed.  
 Moreover, it is clear that for a simple Rangaswamy module $S^q_{\d}=L_K(E)/L_K(E)q^*(\d)$ where $q^*$ is the reciprocal polynomial of $q$, the left ideal $KE^*\cap \ann_{L_K(E)}{\bar v}\, ({\bar v}=v+L_K(E)q^*(\d))$ is exactly
$L^q_{\d}$. Furthermore applying Cunz-Krieger Condition (CK2) successively at vertices of $\d$ one can see that the left ideal $\ann_{L_K(E)}{\bar v}$ can be generated by $L^q_{\d}$ if all vertices of $\d$ are regular. If some vertex of $\d$, say $v$,  is an infinite emitter, then $q^*(\d)$ does not belong to $L_K(E)L^q_{\d}$ whence  $L_K(E)L^q_{\d}$ is a proper subideal of $\ann_{L_K(E)}{\bar v}$. Thus one can summarize the above discussion as
\begin{theorem}\label{ranga} Let $\d=a_1\cdots \a_n (n\geq 1; s(\d)=v)$ be a period of $E$; $q$ a monic irreducible polynomial of degree $l$ in one variable with non-zero constant term. Let $L^q_\d$ be the left ideal of $KE^*$ generated by $q(\d^*)$ and all ghost paths $\neq a^*_m\cdots a^*_1(\d^*)^t$ of length $tn+m \, (1\leq m<n)$ running over all non-negative integers $t$. Then $L^q_\d$ is a maximal left ideal of $KE^*$ which is not open in the topology $\frak T$. If all vertices of $\d$ are regular, then $L_K(E)L^q_\d=L_K(E)q^*(\d)$ and $S^q_\d=\frac{L_K(E)}{L_K(E)q^*(\d)}\cong L_K(E)\otimes_{KE^*} \frac{KE^*}{L^q_\d} = T^q_\d$. If a number $m$ of infinite emitters in $\d$ is positive, then $S^q_\d$ is a proper image of $T^q_\d$. Moreoer $T^q_\d$ has length $ml+1$ with simple subfactors isomorphic to the Cohn-Jacobson modules $S_u$ of multiplicity $l$ where $u$ runs over $m$ infinite emitters of $\d$. 
\end{theorem}
\begin{proof} In view of well-known results on Rangaswamy modules (see e.g. \cite{an} Section 4) it is enough to complete the proof by checking the case when $\d$ contains infinite emitters. First, the right exactness of the tensor product together with the identification of $L_K(E)\otimes_{KE^*}KE^*$ with $L_K(E)$ implies $T^q_\d\cong L_K(E)/L_K(E)L^q_\d$. Put $z_i=(\d^*)^iq^*(\d)$ for $i=0, \cdots, l$, so $z_0=q^*(\d); z_l\equiv 0 \mod L^q_\d$ hold, and let $M_i=L_K(E)z_i+L^q_\d$ for $i=0, \dots, \deg q$. 
Furthermore let $h^*_{ij}=a^*_j\cdots a^*_1z_i$ for $j=1, \dots, n$ whence $h^*_{in}=z_{i+1}$ for $i=0, \dots, l-1$. The obvious equalities $a^*_{j+1}\cdots a^*_1\d^{i+1}=a_{j+2}\cdots a_n\d^i$ for all $0\leq j\leq n-1;\, i\geq 0$ with $a_{j+2}=v$ in the case $j=n-1$; and $x^*a^*_j\cdots a^*_1\in L^q_\d$ for all $a_{j+1}\neq x\in s^{-1}(a_{j+1})$ imply $a_{j+1}h^*_{i, j+1}=h^*_{ij}$ if $s(a_{j+1})$ is a regular vertex. Therefore the images of $h^*_{ij}$ and $h^*_{i, j+1}$ generate the same submodule of $M_i/M_{i+1}$ if $s^{-1}(a_{j+1})$ is a regular vertex where $0\leq j\leq n-1$ and 
$0\leq i\leq l-1$. These results together with the congruences $x^*h_{ij}\equiv 0 \mod L_K(E)h_{i,j-1}$ for all arrows $a_{j+1}\neq x\in s^{-1}(a_{j+1})$, show that $M_i/M_{i+1}\, (i=0, \dots, l-1)$ is a module of length $m$ with composition factors isomorphic to a Cohn-Jacobson module $S_u$ where $u$ runs over infinite emitters of $\d$ and so the proof is complete.
\end{proof}
\begin{remark}\label{rep} The discussion before Theorem \ref{ranga} shows the equality $L_K(E)L^q_\d\cap KE=
(KE)q^*(\d)$, or equivalently, $\d \d^*=v$, when all vertices of $\d$ are regular.  However, $q^*(\d)\notin L_K(E)L^q_\d$ if some vertex of $\d$ is an infinite emitter. 
For a better understanding of Theorem \ref{ranga} we provide another construction of $T^q_\d$ when $\d$ contains at least one infinite emitter, revealing the beauty of representation theory. Simplifying notation{\textcolor{red}{,}} let $v_i=s(t_{\d}(i))$ for $i=0, \dots, n-1$ and $h(i)=h_{\d^l}(i)$ for each $i=0, 1, \dots, n^l$, hence $h(0)=v=v_0, h(1)=a_1, h(n^l)=\d^l$. By the description of $S^q_\d$ up to isomorphism (cf. \cite{an}) one can assume without loss of generality in this case that $\d$ starts from an infinite emitter $v$. Furthermore let $B^q_\d$ be the set of all paths $\m h^*(i)$ in the extended digraph $\hat E$ of $E$ such that $i$ runs over $0, 1, \dots, n^l-1$ and $\m$ runs over paths in $E$, and
$\m h^*(i)$ cannot contain the subpath $a_ia^*_i$ when $v_{i-1}$ is an infinite emitter of $\d$. Let $M$ be the $K$-vector space spanned by the set $B^q_\d$. Then $M$ becomes a left $L_K(E)$-module by mapping $L_K(E)$ homomorphically in the endomorphism ring 
$\End (_KM)$. To reach this goal it is sufficient to assign to each generator of $L_K(E)$, i.e., to each vertex $u$ and each arrow $a$ linear transformations $\ta_u$ and $\ta_a, \ta_{a^*}$, respectively, of $_KM$ such that these linear transformations satisfy relations required in Definition \ref{q2}. However, it is clear how to define these transformations because it is enough to define their values on $B^q_\d$. 
$\ta_u\, (u\in E^0)$ and $\ta_a\, (a\in E^1)$ when $a$ is either $\neq a_i$ or $a=a_i$ and $v_{i-1}=s(a_i)$ an infinite emitter,  is simply a left multipliction with $u$ and $a$, respectively, when it is well-defined, else it is 0.  
For $a=a_i$, and $v_{i-1}=s(a_i)$ a regular vertex, $\ta_{a_i}$ is also a left multiplication by $a_i$ on $\m h^*(j)$ in case $|\m|\geq 1$ but $\ta_{a_i}(h^*(j))=h^*(j-1)$ for $a_j= a_i$ and 0 if $a_j\neq a_i$.  $\ta_{a^*}$
is also a left multiplication, i.e., erasing the arrow $a$ when it is meaningful, i.e., 
$\ta_{a^*}(\m h^*(i))=(a^*\m)h^*(i)$ if $|\m|\geq 1$ but 
$$\ta_{a^*}(h^*(i))=\begin{cases} 0 & \text{if $h(i)a\neq h(i+1)$ and $1\leq i+1\leq n^l$},\\ h^*(i+1) & \text{if $h(i)a=h(i+1)$ and $1\leq i+1\leq n^l-1$},\\ (\d^*)^l-q(\d^*)  & \text{if $i=n^l-1$ and $a=a_n$}. \end{cases}\,\,$$
It is a routine task to verify that $\ta$ is a well-defined $K$-algebra homomorphism from  $L_K(E)$ into $\End(_KM)$ making $M$ a left $L_K(E)$-module isomorphic to $T^q_\d$. Moreover, the image of $q$, i.e., $\t_q$ is not zero, that is, $q^*(\d)\notin L_K(E)q(\d^*)$ holds. A similar circumstance appears when the Cuntz-Krieger condition (CK2) is ignored, i.e., when we consider the Cohn path algebra $C_K(E)$ instead of the Leavitt path algera $L_K(E)$. For the sake of clarity consider the Jacoson algebra $K\langle x, y\rangle$ subject to $yx=1$, i.e. the Cohn path algebra of the \emph{rose $R_1$ with 1 pedal} graph having one vertex 1 and one loop $x$ with $y=x^*$. Then the vector space $V$ generated by the basis $B=\{1, x, x^2, \cdots \}$ can be considered in the obvious manner as a representation of $C_K(R_1)$ by assigning to $1$ the identical linear transformation on $V$, to $x$ the left multiplication on $B$ and to $y=x^*$ a linear transformation $v\mapsto kv,\, x^{l+1}\mapsto x^l\, (l\geq 0)$ where $k$ is an arbitrary, not necessarily nonzero scalar from $K$. Moreover, all such representations of the Jacobson algebra are isomorphic one to the other and also to a minimal one-sided, here, left ideal. Such a representation is, however,  not one of $L_K(R_1)$ because the Cuntz-Krieger condition (CK2) fails in the latter case! This is why we call the representation $S_v$ defined by an infinite emitter $v$ a Cohn-Jacobson representation of $L_K(E)$ in \cite{an}. The picture becomes complete by giving a representation space $_KM$ of $S^q_\d$ when all vertices of $\d$ are regular.  For the sake of clarity after a normalization if it is necessary, we write $q^*(x)=kx^l-q_1(x)=1-xp(x);\, p(x), q_1(x)\in K[x]; 0\neq k\in K$ and $t_0=v=t_{\d^l}(n^l-0), t_1=a_n=t_{\d^l}(n^l-1),\dots, t_i=t_{\d^l}(d^l-i)\, (i=0, 1, \dots, n^l-1)$. Let $B^{q^*}_\d$ be the set of all  paths $\m t_i$  where $i\in \{0, \dots, n^l-1\}$ and $\m\in F(E)$ such that the composition path $\m t_i$ is well defined and the tail $t_{\m}(|\m|-1)t_i\neq t_{i+1}$, and $M$ the $K$-vector space generated by the basis $B^{q^*}_\d$. Therefore
$\m\in F(E)v$ belongs to $B^{q^*}_\d$ if and only if $t_{\m}(|\m|-1)\neq a_n$ or equivalently, $\m$ cannot be written in the form $\m=\n a_n (\n\in F(E))$. In other words, $B^{q^*}_\d$ is exactly a Schreier basis of the left ideal $KEq^*(\d)$ of $KE$ with respect to its standard basis $F(E)$. Every vertex $u\in E^0$ defines the linear transformation $\ta_u$ on $_KM$ by left multiplication, i.e. by concatenation on the left side with $u$ with paths from $B^{q^*}_\d$. Each arrow $a\in E^1$ defines two linear transformations $\ta_a, \ta_{a^*}$ on $_KM$ as follows 
$$\ta_a(\m t_i)=\begin{cases} 0 & \text{if 
$r(a)\neq s(\m t_i)$},\\ a\m t_i & \text{if $r(a)=s(\m t_i)$ and $a\m t_i\neq \d^l$},\\
k^{-1}q_i(\d) & \text{if $|\m|=0; i=n^l-1$ and $a=a_1$}, \end{cases}\,\,$$
and
$$\ta_{a^*}(\m t_i)=\begin{cases} 0 &  \text{if $a\neq h_{\m t_i}(1)$ or $s(a)=v\,\, \&\,\, \m t_i=v$},\\
t_{\m t_1}(1) & \text{if $a=h_{\m t_i}(1)$},\\ a_2\cdots a_np(\d)
 & \text{if $\m t_i=v$, i.e., $\m=v; i=0$}. \end{cases}\,\,$$
It is a routine but useful to check that the assignment $\ta$ induces an algebra homomorphism from $L_K(E)$ into $\End(_KM)$ making $M$ a representation space of of $L_K(E)$ isomorphic to $S^q_\d$.
\end{remark}  
Since $L^q_{\d}\cap K[\d^*]$ is a maximal ideal of $K[\d^*]$ which is, in general, not a factor of the subalgebra $v(KE^*)v$ one obtains immediately that Rangaswamy modules do not, in general, coincide with
Hilbert modules. Moreover, one can clearly see difficulties arising when infinite emitters are part of the picture.  $KE^*$ admits both maximal left ideals which are not open in $\frak T$ and not maximal left ideals which are maximal in the set of not open left ideals such that these left ideals generate maximal left ideals of $L_K(E)$ when $E$ is a finite digraph but not necessarily ones when infinite emitters are present. Furthermore a careful examination of the discussion above and Schreier technique lead to a method for constructing new classes of not necessarily irreducible representations of $L_K(E)$. 
\vskip 0.3cm
First, we fix some notation. For a vertex $v\in E$, we simply write ${\frak c}_v=\frak c $ for a set of closed paths based at $v$. Since a structure of closed paths is very complicated,  we assume that elements in ${\frak c}$ are periods. According to a standard convention, a subalgebra of either $KE$ or $KE^*$ generated by $\frak c$ or ${\frak c}^*$ is denoted by $K\frak c$ or $K\frak c^*$, respectively. Consequently, one can assume without loss of generality that no path of $\frak c$ can be written as a product of the other ones from $\frak c$. Moreover,
 the $*$-subalgebra of $L_K(E)$ generated by $\frak c$ is denoted by  $J(\frak c)$. If $\frak c$ consists of a single period $\d$, then $J(\frak c)$ is denoted simply by $J(\d)$ which is called a \emph{Jacobson} or a \emph{Toeplitz subalgebra} defined by a period $\d$ in \cite{an}. We are now in position to describe a method for constructing irreducible representations of $L_K(E)$ together possibly with their endomorphism rings.
Inspired by both the argument on Rangaswamy modules and Remarks \ref{hilbert0}(2), one can assign to $\frak c$ a left ideal $L$ of $KE^*$ as follows. Let $B_{\frak c}$ be the set of all paths of the form $\m^* \n^*$ where $\m$ runs over all proper heads of some $\d \in \frak c$ and $\n$ is an arbitrary element in the semigroup generated by $\frak c$. Then $B_{\frak c}$ is  exactly a Schreier basis of the left ideal $L$  generated by $F^*(E)\setminus B_{\frak c}$ with respect to the basis $F^*(E)$ of $KE^*$ \cite{ajm}. Recall that a Schreier basis of a left ideal $L$ with respect to the basis $F^*(E)$ of $KE^*$  is by definition \cite{ajm} a subset  $B$ of $F^*(E)$ closed under taking tails such that $B$ generates a $K$-subspace of $KE^*$ directly complementing $L$. However, this left ideal is not maximal, or more generally, maximal in the set of not open left ideals with respect to the topology $\frak T$. However, fortunately enough, one can extend $L$, by adding an appropriate left ideal $G$ of $K\frak c^*$ to a left ideal $L_{\frak c}=L+(KE^*)G$ of $KE^*$. 
Since one can drop out periods which are contained in $G$ one can assume without loss of generality that $G$ is disjoint from the semiroup generated by $\frak c$. 
It turns out that $L_{\frak c}$ is usually maximal in the set of left ideal of $KE^*$ which are not open in the topology $\frak T$. Hence $T_{\frak c}=L_K(E)\otimes_{KE^*} KE^*/L_{\frak c}$ is an irreducible representation of $L_K(E)$ in certain cases. We shall call irreducible representations of $L_K(E)$ gotten in this way \emph{Mantese-Rangaswamy simple modules} acknowledging the idea of both Mantese (in her private communication) and Rangaswamy. A most obvious example employing this method is found when we consider left ideals of $K\frak c^*$ defined by irrational Chen words in periods from $\frak c$ instead of arrows. Furthermore one can take $G$ a maximal left ideal of $K\frak c^*$ such that the largest ideal $P$ contained in $G$ is maximal or even primitive ideal with the "nice" factor algebra $K\frak c^*/P$. In particular, one of the most interesting cases is that when $G$ itself is a maximal two-sided ideal and $K\frak c^*/G$ is a division ring. Moreover, one can take into account the fact that $K\frak c^*$ is usually
a free associative algebra whence one can apply Lewin-Schreier technique to obtain certain sets of generators for $G$ as both a two-sided or a left ideal, respectively. This allows the use of the well-developed methods, ideas and results of finite dimensional quiver algebras possibly with relations in the study of the module theory of Leavitt path algebras. This is extremely important when $v$ is an infinite emitter. 
Since  $L_K(E)L_{\frak c}$ is generally not a maximal left ideal, one needs a maximal left ideal containing it. It seems that one can get such a left ideal when we replace relations defining $G$ by elements from $KE$. It is worth to note that the above described process reduces to a construction of Hilbert modules, for example, when $\frak c$ consists of loops at $v$. However, if $\frak c$ contains a period of length $\geq 2$, then one obtains definitely a Mantese-Rangaswamy module which is not a Hilbert module!
\vskip 0.3cm
One can summarize the method described above more generally as follows. First, in view of flat localizations, one can pick up a vertex $v$ such that $vKE^*v$ contains free subalgebras and then consider a maximal left ideal $L$ of $vKE^*v$ and its extension $L_K(E)L$. There are two possibilities. The first case is the lucky one when $L_K(E)L$ is maximal and then it induces an irreducible representation of $L_K(E)$. The second case, as we have seen in Theorems \ref{irr} and \ref{ranga}, is a mixed case when there are digraphs such that $L_K(E)L$ is either maximal or not maximal, respectively. In this case the proper representation $T=L_K(E)/L_K(E)L$ may have interesting properties.

For some applications of the process described above, we now present some cases by way of demonstration.
\begin{theorem}\label{qua1} Let $\frak c$ be a set of periods based at $v\in E^0$  such that sets of arrows in each period in $\frak c$ are pairwise disjoint, and $B_{\frak c}$  the set of all paths of the form $\m^* \n^*$ where $\m$ runs over all proper heads of $\d \in \frak c$ and $\n$ runs over all products of $\d \in \frak c$.  If $D$ is a division ring factor of $K\frak c^*$ by an ideal $G$ such that every $\d \in \frak c^*$ has a nonzero image in $D$, and  $L_{\frak c}$ is a left ideal of $KE^*$ generated by $G$ and paths outside $B_{\frak c}$, then $L_{\frak c}$ is a maximal left ideal of $KE^*$ which is not open in the topology $\frak T$ whence for finite digraphs $E$ $T_{\frak c}=\frac{L_K(E)}{L_K(E)L_{\frak c}}\cong L_K(E)\otimes_{KE^*}\frac{KE^*}{L_{\frak c}}$ is an irreducible representation of $L_K(E)$ whose endomorphism ring is isomorphic to $D$.
\end{theorem} 
\begin{proof} Without loss of generality one can assume that $\frak c$ has at least two periods because the case of one period reduces to the well-known case of Rangaswamy modules.The idea of the proof is the following simple but crucial observation. Namely, by the assumption on arrows in $\frak c$ and the definition of $L_{\frak c}$ for any element $\n$ of the semigroup generated by $\frak c$ and any proper head $h$ of a period
$\d \in \frak c$, one has $a^*h^*_j\n^*\in L_{\frak c}$ for every arrow $a\in E^1$ which is not an arrow of $\d$. First we show that $L_{\frak c}$ is a proper left ideal of $KE^*$. Let $L$ be the left ideal generated by $F^*(E)\setminus B_{\frak c}$. Then we have the direct decomposition $KE^*=V\oplus L$ where $V$ is the $K$-space spanned by $B_{\frak c}$. Assume indirectly the equality $L_{\frak c}=KE^*$. Then there are elements $g, g_i\in K\frak c^*$ and proper heads $h_i$ of periods from $\frak c$ such that the following congruence
 $$v=g+\sum h^*_ig_i  \mod L$$
hold. By the observation made at the beginning of the proof and the extra assumption one has for two different periods $\d_1, \d_2$ the congruence
$$\d^*_2d^*_1=\d^*_2\d^*_1g\in G \mod L$$
which is a contradiction because elements of the semigroup generated by $\frak c^*$ maps to nonzero elements of $D$ by assumption. Therefore $L_{\frak c}$ is a proper left ideam of $KE^*$.

Next we are showing that $L_{\frak c}$ is a maximal left ideal of $KE^*$. Consider an arbitrary $r\in KE^*\setminus L_{\frak c}$. Then by the definition of $L_{\frak c}$ there are $\d_i \in \frak c (i=1, \dots, l)$, proper heads $h_{ij}\, (1\leq j\leq |\d_i|-1)$ of $\d_i \,(i=1, 2)$ of increasing length and nonzero elements $t_{ij}\in K\frak c^*$ together with  $t_0\in K\frak c^*$ nonzero $\mod L_{\frak c}$ such that the following congruence
\begin{align}\label{eq2}
r \equiv t_0+\sum h^*_{ij}t_{ij}=t_0+r_0 \quad \mod L_{\frak c}
\end{align}
holds with the least number of summands of $t_0, r_0$ and subject to this condition, of the minimal degree as an element of $KE^*$. Then any subsum of \eqref{eq2} does not belong to $L_{\frak c}$. 
If there is a period $\d\in \frak c$ such that none of the $h_{ij}$ is its head, then $\d^*r=\d^*t_0$
maps to a nonzero element in $D$. Therefore one can assume with thout loss of generality that for every period $\d$ in $\frak c$ there is some of its heads belonging to the $h_ij$. Take an arbitrary period $\g \in \frak c$ and one of its heads, say $\a$ with $\g=\a \b$, such that $\a$ is one of the $h_{ij}$. Then $\b^*r$ is again non-zero $\mod L_{\frak c}$. Since $\b^*t_0\in L_{\frak c}$ and $\b^*h^* t_{ij}\in L_{\frak c}$ hold for any head of $\d \in \frak c; \d\neq g$ we have 

$$\b^*r= t+\sum h^*_lt_l \mod L_{\frak c}$$
where $t, t_l\in K\frak c^*$ but $t, t_l\notin G$ and $h_l$ are proper heads of $\g$. Since the case $|\frak c|=1$ is exactly the case of Rangaswamy simple module defined by an irreducible polynomial and a set $\frak c$ of one period $\d$ one can assume that $\frak c$ has at least 2 periods. In this case, there is  a period $\l \in \frak c; \l \neq \g$ whence $\l^*\b^*r=\l^* t\in K\frak c^*$ is not zero $\mod G$. 
Consequently,  $L_{\frak c}$ is a maximal left ideal of $KE^*$ because $D$ is a division ring and $\l^*\b^*r$ becomes nonzero in $D$. All the other claims of Theorem \ref{qua1} are immediate by the assumption and Theorem \ref{simple1}
\end{proof}
\begin{remarks}\label{warning2} 
\begin{enumerate}
\item It is difficult to test whether a maximal left ideal of $KE^*$ generates a maximal left ideal of $L_K(E)$. For finite digraphs as a consequence of flat localizations it is enough to check whether a left ideal of $KE^*$ is not open in the topology $\frak T$ defined by powers $I^n$ of the ideal generated by ghost arrows and sinks of $E$.  
When infinite emitters are involved, one can use representation theory to decide the question, as for infinite Chen words in the discussion of Theorems \ref{irr} and \ref{ranga}. However, a maximal left ideal $L$ of $KE^*$ satisfying $v\notin L$ where $v\in E^0$ is an infinite emitter, generates a proper left ideals of $L_K(E)$ because in this case linear combbinations of paths $\m\n^*\,( s(\m)=s(\n)=v, r(\m)=r(\n))$ do not produce $v$ by the lack of Cuntz-Krieger condition (CK2). Another way to see this claim is by observing that the left ideal of $KE^*$ generated by all ghost arrows and singular vertices of $E$ form a left dense ideal of $ KE^*$ which allows us to view $L_K(E)$ as a subring (ring of quotients in Utumi's sense) of the maximal left ring of quotients of $KE^*$, and then to look at the corresponding images of $v$ and linear combinations  $\sum_i k_i\m_i\n^*_i\,( s(\m_i)=s(\n_i)=v, r(\m_i)=r(\n_i); \, 0\neq k_i\in K)$ at $v$ in $KE^*$.

\item In connection to Remarks \ref{hilbert0}(2) the construction of maximal left ideals $L_{\frak c}$
shows bona fides of localization technique when we take a careful look at intersections of maximal left ideals of $L_K(E)$ and $KE^*$. They lead to relatively simple, interesting way constructing maximal left ideals of quiver algebras probably also with endomorphism rings of the associated irreducible representations.
\item As in the case of finite digraphs instead of considering maximal ideals $G$ of $K\frak c^*$ whose factors are division rings one can take first a simple modules $K\frak c^*/L$ of $K\frak c^*$ such that images of periods from $\frak c^*$ in the factor algebra $K\frak c^*/G$ are nonzero where $G$ is the primitive annihilator ideal, i.e., the largest ideal contained in $L$. Then consider the left ideal $L_{\frak c}$ of $KE^*$ generated by $F^*(E)\setminus B_{\frak c}$ and $L$. Note that $L_{\frak c}$ depends on both $\frak c$ and $L$. We deal with division algebra factors for the sake of simplicity and transparence. 
\end{enumerate}
\end{remarks}
The following important but now obvious consequence of Theorem \ref{qua1} shows clearly that the class of Mantese-Rangaswamy irreducible representations are essentially different from the Hilbert ones even in the case of finite digraphs. For the sake of simplicity recall that two closed paths are called \emph{disjoint} if they have no common arrows. Consequently, $K\frak c^*$ is a free associative algebra with the identity $v$ of rank $|\frak c^*|$ if $\frak c$ is a set of disjoint periods ased at a vertex $v$. Moreover $\frak c$ is finite if $E$ is finite.
\begin{corollary}\label{manqua1} Let $E$ be a finite digraph. If $\frak c$ is a set of disjoint periods at a vertex $v$ and $D$ is a division $K$-algebra of $K\frak c^*$ by an ideal $G$ disjoint from $\frak c^*$, then  the left ideal $L_{\frak c}$ of $L_K(E)$ defined in Theorem \ref{qua1} is maximal and the simple factor module $L_K(E)/L_{\frak c}$ is a Mantese-Rangaswamy irreducible representation of
$L_K(E)$ with the endomorphism ring $D$ which is not a Hilbert module if one period in $\frak c$ is not a loop.
\end{corollary} 
In the remainder of this section for the sake of transparence we consider only loops as periods although one can see by almost the same proof but with more technical, unnecessarily complicated details that the results remain true for disjoint periods which are no longer loops. Therefore  
we are going to present examples for arbitrary digraphs $E$ together with maximal left ideals $L$ of $KE^*$ such that the extended left ideals $L_K(E)L$ are maximal and the induced simple modules $L_K(E)/L_K(E)L$ have noncommutative endomorphism rings. For this aim assume that the characteristic of $K$ is not 2. A \emph{(generalized) quaternion algebra} $D=(c, d)$ where $c, d$ are scalars in $K$ is a 4-dimensional $K$-algebra  with a basis $\{1, i, j, k\}$ subject to $i^2=-c, j^2=-d, ij=-ji=k$. It is well-known that $D$ is a centrally simple $K$-algebra and even a division ring if the associated quadratic form $Q(x_0, x_1, x_2, x_3)=x^2_0+cx^2_1+dx^2_2+cdx^2_3$ does not represent $0$, i.e., $0$ is the unique root of $Q$. Since our goal is getting representations with division quaternion endomorphism algebras we assume that $D$ is a division ring although it works generally with a slight modification. In fact, in the same manner one can construct irreducible representations whose endomorphism rings are cyclic, even finite dimensional division algebras over $K$, establishing a connection to finite dimensional division rings. 

\begin{theorem}\label{qua2} Assume that there are two 2 loops $a, b$ based at $v\in E^0$ and $K$ is a field of characteristic not 2 together with non-zero scalars $c, d$ such that the associated quaternion $K$-algebra
$D=(c, d)$ is a division ring. If $L_{c, d}$ is the left ideal of $KE^*$ generated by $\{u\in E^0\,|\, u\neq v\}\cup \{x^*\, |\, x\in s^{-1}(v)\,\, \& \,\, a\neq x\neq b\}\cup \{(a^*)^2+cv, (b^*)^2+dv, a^*b^*+b^*a^*, b^*a^*b^*+da^*, a^*b^*a^*+cb^*\}$ and all ghost paths not contained in the semigroup generated by $a^*, b^*$, then $T_{c, d}=L_K(E)/L_K(E)L_{c, d}$ is an irreducible representation of $L_K(E)$ with $\End(T_{c, d})\cong D=(c,d)$. 
\end{theorem}
\begin{proof} The definition of $L_{c, d}$ shows that $\{v, a^*, b^*, a^*b^*\}$ induces a $K$-basis of $KE^*/L_{c, d}$. Therefore in view of the Lewin-Schreier formula and Shreier's technique, \cite{rr} and \cite{le1} the elements $(a^*)^2+cv, (b^*)^2+dv, a^*b^*+b^*a^*, b^*a^*b^*+da^*, a^*b^*a^*+cb^*$ form a free basis of $\langle(a^*)^2+c, (b^*)^2+d, a^*b^*+b^*a^*\rangle \vartriangleleft K\langle a^*, b^*\rangle$ as a left $K\langle a^*, b^*\rangle$-module and the associated factor algebra is the quaternion division algebra $D=(c,d)$. It is worth to note that according to Remarks \ref{hilbert0}(2) one cannot replace the relations  $(a^*)^2+cv, (b^*)^2+dv$ by requiring $ca^2+db^2+v$ even in the case of finite $E$.
Moreover, it is immediate that every non-zero element of $T_{c, d}$ can be represented  by a linear combination of paths from $B=F(E)v\cup F(E)a^*\cup F(E)b^*\cup F(E)a^*b^*$, i.e., an element $z$ of the vector space $V$ spanned by $B$. In the other words, one can assign to each generator $u\in E^0, c\in E^1$ linear transformations $\ta_u, \ta_c, \ta_{c^*}$ of $V$ in the obvious manner and to check that this assignment defines an $K$-algebra homomorphism from $L_K(E)$ in 
$\End(_KV)$. By this way, one can see again that $L_{c, d}$ and $L_K(E)L_{c,d}$ are proper left ideals of $KE^*$ and $L_K(E)$, respectively.  
First, we show the simplicity of $T_{c,d}$. Take an arbitrary non-zero element $m$ of $T_{c,d}$ which can be represented by some $z\in V$. The \emph{width} $N$ and \emph{degree} $\deg z$ of $z$ are by definition the number and the longest length of paths from $F(E)v$ involved in the linear combination of elements from $B$ for $z$. Among the  representatives of $m \mod L_K(E)L_{c, d}$ choose $V\ni z \equiv m \mod L_K(E)L_{c, d}$ having the smallest width, and that subject to this condition, the degree
$\deg z$ is minimal. Therefore any proper subsum of $z$ is also not contained in $L_K(E)L_{c, d}$.   
To check the simplicity of $T_{c,d}$ it is enough to see that there is a path $\a\in F(E)$ such that $k\a^*z\in \{v, a^*, b^*, a^*b^*\}$ for some $k\in K$. We shall use double induction on the width and length of $z$. The case $N=1$ is trivial. Assume the claim for all widths $\leq N-1 (N\geq 2)$ and all degrees. Consider now a nonzero element $m$ of $T_{c,d}$ such that $m$ can be represented by a $z\in V$ with the smallest width $N$ and that subject to this condition, the degree
$\deg z$ is minimal. By considering $\a^*z$ where $\a\in F(E)v$ is a path of $z$ having minimal length, one can assume that $z$ contains at least one element from $\{v, a^*, b^*, a^*b^*\}$. If an initial arrow $x$ of a path of $z$ is not in $\{a, b\}$, then $x^*z$ would have the smaller width and so the claim follows by the induction hypothesis. Moreover, the claim is also obvious if $\deg z=0$. Hence, assume $\deg z\geq 1$. Therefore $z$ has at least one path $\a$ of minimal positive length. Take $x\in \{a, b\}\setminus h_\a(1)$, then $x^*z\neq 0$ has the smaller width, hence the claim follows from the induction hypothesis. It is now easy to compute the endomorphism algebra $\End (T_{c,d})$. Every endomorphism of  $T_{c,d}$ is uniquely determined by the image $z$ of $v\in V$. By the above argument there is $\a \in F(E)v$ satisfying $\a^*z\in \{v, a^*, b^*, a^*b^*\}$. This means that every endomorphism of  $T_{c,d}$ is induced by one of $KE^*/L_{c,d}$. However, the endomorphism algebra $\End (KE^*/L_{c,d})$ is the quaternion division algebra
$D=(c, d)$ and so our proof is complete. 
\end{proof}
\begin{remarks}\label{warning1}
\begin{enumerate}
\item Results of this sections show that representation theory of Leavitt path algebras can start first with one of finite digraphs and then one can focus on left ideals of the dual quiver algebras which are maximal in the set of not open ideals with respect to the  topology $\frak T$ defined by the ideal generated by all ghost arrows and sinks. The corresponding extended left ideals of Leavitt path algebras may induce not necessarily simple representations with interesting important properties. Looking for generators of these left ideals may lead to interesting elements of the quiver algebras which could generate maximal left ideals in the Leavitt algebras (cf. also Remark \ref{chen}, Remarks \ref{hilbert0}(2)).
\item  If $E$ is finite, conditions of Theorem \ref{qua1} imply that $\frak c$ is finite whence $K\frak c^*$ is a free associative algebra in $|\frak c|$ variables. To remove this finite restriction one can consider free associative subalgeras of $K\frak c^*$ in the case $|\frak c^*|\geq 2$ because $K\frak c^*$ contains in this case free associative algebras of infinite rank. Therefore one can consider division $K$-algebras of these free associative subalgebras, and then use them to obtains simple modules over $KE^*$! 
\end{enumerate}
\end{remarks}
Our next goal is to extend Mantese's result, Theorem \ref{man} to digraphs having infinite emitters. Even for these digraphs we still have essentially a Hilbert representations because only linear polynomials and loops are involed in the construction.  Keeping the notation fixed above, we have

\begin{theorem}\label{qua3} If $a_i\,(i\in I, |I|\geq 2)$ is a not necessarily finite set  of loops at $v$ and $f_{\frak r}=\sum_i r_ia_i-v$ for $0\neq \frak r \in K^I$ with at least two but almost all entries 0, then $S^K_{\frak r}=L_K(E)v/L_K(E)f_{\frak r}$ is a simple image of $T^K_{\frak r}=L_K(E)\otimes_{KE^*}\frac{KE^*v}{L_{\frak r}}$, and $\End(S^K_{\frak r})\cong K$ where $L_{\frak r}$ is a left ideal of $KE^*$ generated by $r_iv-a^*_iv \,(i\in I)$ and $\{x^*\in s^{-1}(v) \,|\, x\neq a_i \,\, \forall \,\, r_i\neq 0 \}$ and $L_{\frak r}=L_K(E)f_{\frak r}\cap KE^*$ holds, provided $v$ is regular.
Moreover, if $v$ is an infinite emitter, then $r=f_{\frak r}\otimes \bar v\, (\bar v=v+L_{\frak r}\in T^K_{\frak r})$ is not zero; the submodule
$L_K(E)r$ is simple and isomorphic to the Cohn-Jacobson module $S_v$ whence $T^K_{\frak r}$ has length 2.
\end{theorem}
   
\begin{proof} First we identify  $T^K_{\frak r}$ with $L_K(E)v/L_K(E)L_{\frak r}$ by right exactness of tensor products. $L_K(E)L_{\frak r}\subseteq L_K(E)f_{\frak r}$ holds clearly. Next, we show that $L_K(E)L_{\frak r}$ and $L_K(E)f_{\frak r}$ are proper different left ideals of $L_K(E)$ by using representation theory when $v$ is an infinite emitter.  Let $M=(KE)v$ be a $K$-vector space generated by the basis $B=F(E)v$. Each vertex $u\in E^0$ or each arrow $a\in E^1$ defines a linear transformation  $\ta_u$ or $\ta_a$ on $M$ by left multiplication with $u$ or $a$, respectively. Each arrow $a$ defines a linear transformation $\ta_{a^*}$ on $M$ as follows
$$\ta_{a^*}(\m)=\begin{cases} 0 &  \text{if $a\neq h_{\m }(1)$ or $s(a)=v; a\notin \{a_1, \dots, a_n\}\,\, \&\,\, \m=v$},\\
t_{\m}(1) & \text{if $a=h_{\m}(1)$},\\ r_iv
 & \text{if $\m=v\,\, \& \,\, a=a_i $}. \end{cases}\,\,$$
It is a routine exercise to verify that the assignment $\ta$ makes $M$ a left $L_K(E)$-module, $\ann_{L_K(E)}v=L_K(E)L_{\frak r}$ and $f_{\frak r}\notin L_K(E)L_{\frak r}$ whence the claim follows.
For a regular vertex $v$ one can see $L_{\frak r}=L_K(E)f_{\frak r}\cap KE^*$ in the same way as in Remark \ref{chen} using the Krieger-Cuntz relation (CK2). In this case one can take the set $B=F(E)v\bigcup\limits_{i=1}^{n-1} F(E)a_i$ and the vector space $M$ generated by the basis $B$ to obtain a representation space of $S^K_{\frak r}$ by requiring
$$\ta_{a_n}(v)=r^{-1}_n(v-\sum\limits_{i=1}^{n-1}r_ia_i),$$
whence $ L_K(E)L_{\frak r}$ is a proper left ideal of $L_K(E)$ as claimed.
 
Finally to verify a simplicity of $S^K_{\frak r}$ one observes that
every nonzero element $m\in S^K_{\frak r}$ can be  represented by $t=tv=\sum k_j\m_j\n^*_j \mod L_K(E)f_{\frak r}$ where $k_j\in K$ and $\m_j, \n_j$ are real paths ending in $v$. The assumption implies that all $\n_j$ are paths in the $a_i$. Furthermore because $a^*_iv-r_iv \in L_K(E)L_{\frak r}$ for each index $i$, one can assume $m\equiv \sum_j k_j\m_j=t \mod L_K(E)L_{\frak r}$ with a possibly minimal number of summands, and subject to this condition the degree $\deg t=\max_j |\m_j|$ is minimal. Therefore any subsums of $t$ are not contained in $L_K(E)L^K_{\frak r}$. By multiplying with some $k\m^* (0\neq k\in K)$ of a minimal length, one can assume without loss of generality that $m\equiv t=-1+\sum_j k_j\m_j$ with a possibly minimal number of terms, and subject to this condition the degree $\deg t=\max_j|\m_j|$ of $r$ is minimal. We show by double induction that there is some $s\in KE^*$ satisfying $st\equiv v \mod L_K(E)f_{\frak r}$.

The case $\deg t=0$ obviously holds. Let $\deg t=1$. Then $\m_j\in s^{-1}(v)$. If there is $\m_j=a_i$ with $r_i=0$, then $k^{-1}_ja^*_it=v$ holds. Therefore one can assume that the $\m_j$ form a subset of $a_i$ with $r_j\neq 0$. 
If there is some $a_i$ with $r_i\neq 0$ different from all $\m_j$, then $-r_i^{-1}a^*_it=-r^{-1}_ia^*_iv \equiv v \mod L_K(E)f_{\frak r}$ holds. Hence one can assume that the $\m_j$ form the set of all $a_i$ with $r_i\neq 0$. Then for each such $a_i$ one has $a^*_it=(-a^*_i+k_i)v\notin L_K(E)f_{\frak r}$ if and only if $k_i\neq r_i$. The case that $k_i=r_i$ for every index $i$ cannot happen because it would imply $t=f_{\frak r}\equiv 0 \mod L_K(E)f_{\frak r}$, a contradiction. Thus the case $N=1$ is verified. 

Assume now the claim for all degrees $\leq N-1 (N\geq 2)$. Let $l=\min_j |\m_j|$. Then $l\geq 1$. If there is a monomial $\n$ containing a factor $a_i$ with $r_i=0$ of length $l$ among the $\m_j$, i.e.,
$\n=\n_1a_i\n_2${\textcolor{red}{,}} then $a^*_i\n^*_1t\notin L_K(E)f_{\frak r}$ of degree $\leq N-1$ with fewer terms. Hence one can assume that the $\m_j$ of length $l$ form a subset of monomials with factors $a_i$ satisfying $r_i\neq 0$. Consequently, for a monomial $\m$ of length $l$ among the $\m_j$ with coefficient $k$ one has $\m^*t=(k-\m^*)v+\cdots\notin L_K(E)f_{\frak r}$ if $k\neq r_{i_1}\cdots r_{i_l}$ if $\m=a_{i_1}\cdots a_{i_l}$. However, the case when the $\m_j$ form the set of all monomials in the $x_i$ and $k_j=r_{i_1}\cdots r_{i_l}$ for $\m_j=x_{i_1}\cdots x_{i_l}$ is impossible because
it would imply for $f=f_{\frak r}+v$ that
$$\sum\limits_{|\m_j|=l}k_j\m_j-v=(f^{l-1}+\cdots +f+1)f_{\frak r}\equiv 0 \mod L_K(E)f_{\frak r},$$
 a contradiction in view of the fact that the subsum of $m$ consists of terms of length $l$  not congruent to zero. This shows that there is $s_1\in  KE^*$ such that $s_1m\equiv s_1t \mod L_K(E)f_{\frak r}$ and $s_1t$ has degree $\leq N-1$ and so the proof is complete by the induction hypothesis. 

 It remains only to verify the claim concerning $L_K(E)r$ when $v$ is an infinite emitter. We have already seen $r\neq 0$. Since $a^*_ir=0$ for every $i\in I$ and so $b^*r=0$ holds for every $b\in s^{-1}(v)$, $L_K(E)r$ is isomorphic to the Cohn-Jacobson module  $S_v$, hence it is simple.  
\end{proof}
I do not know, however, whether $T^K_{\frak r}$ is semisimple if it is not simple.

The next result shows a wide applicability of the method described before Theorem \ref{ranga}. In contrast to Theorems \ref{qua1}, \ref{qua2} and \ref{qua3} the sets of arrows in periods are not necessarily disjoint and we need not take division algebra factors of $K\frak c^*$. It is enough to consider simple modules over $K\frak c^*$. Therefore Proposition \ref{linear} provides an example of Mantese-Rangaswamy irreducible representation which is not a Hilbert one. Moreover, both representation theory of Leavitt path algebras and Schreier techniques developed in \cite{ajm} and \cite{a1} are used heavily in the proof.
\begin{proposition}\label{linear} If  $a, b$ are loops at $v\in E^0,\, f=v-ab-b$ and $L_{ab}=L_K(E)f\cap KE^*$, 
then $T_{ab}=L_K(E)v/L_K(E)L_{ab}\cong L_K(E)\otimes_{KE^*} KE^*v/L_{ab}$
is an irreducible representation of $L_K(E)$ and $\End(T_{a})\cong K$ if $v$ is regular. If $v$ is an infinite emitter and $\bar f= f+L_K(E)L_{ab}\in T_{ab}$, then $0\subseteq L_K(E)a^*{\bar f}\subseteq L_K(E){\bar f}\subseteq T_{ab}$ is a composition chain of length 3, having the first two subfactors isomorphic to Cohn-Jacobson module $S_v$ and the third one is
$S_{ab}$  with $\End(S_{v})\cong K$. 
\end{proposition}
\begin{proof}  It is well-known that the subalgebra $K\langle a^*, b^*\rangle$ of $KE^*$ is a free associative $K$-algebra whence the subalgebra $K\langle a^*, b^*a^*\rangle$ is also free on the set $\{a^*, b^*a^*\}$. From the definition of $L_{ab}$ the set $\{v, a^*\}$ induces a $K$-basis of $KE^*/L_{ab}$. Therefore the relations $a^*-(a^*)^2, b^*a^*-v$ define only a simple module of dimension 2 over $K\langle a^*, b^*a^*\rangle$ and $Kv+Ka^*$ is a $K$-space directly complementing the corresponding left ideal $L_{ab}\cap K\langle a^*, b^*a^*\rangle$. Consequently, $L_{ab}$ is a maximal left ideal of $KE^*$. Next, we compute the endomorphism ring of $KE^*/L_{ab}$. An endomorphism of $KE^*/L_{ab}$ is completely determined by an image of $v+L_{ab}$ which can be represented by some linear combination $k_1v+k_2a^* \mod L_{ab}$. If $k_2\neq 0$, then $k_1v+k_2a^* \mod L_{ab}$ is not annihilated by $b^*$ whence the claim. Moreover, we claim the inclusion $L_K(E)L_{ab}\subseteq L_K(E)(v-a-ab)=L$.  Namely, for every $x\in s^{-1}(v)\setminus \{a\}$ one has $x^*(v-a-ab)=x^*\in L$. Furthermore $(a^*)^2(v-a-ab)=(a^*)^2-a^*\in L$ as well as $b^*a^*(v-a-ab)=b^*a^*-v\in L$. The equality $a^*(v-a-ab)=a^*-v-b$ implies for $x\in s^{-1}(v)\setminus \{a, b\}$ the equality $x^*a^*(v-a-ab)=x^*a^*-x^*$ whence $x^*a^*\in L$ holds. Therefore one obtains the inclusion $L_{ab}\subseteq L$. Furthermore we have the equality 
 $L_K(E)L_{ab}=L_K(E)(v-a-ab)$ if $v$ is regular. Namely, the regularity of  $v$ implies
$$v=\sum_{x\in s^{-1}(v)}xx^*v\equiv aa^*v=aa^* \mod L_K(E)L_{ab}, $$ 
hence
$$v\equiv aa^*v\equiv a\Bigl (\sum_{x\in s^{-1}(v)}xx^*\Bigr ) a^*v\equiv a(aa^*+bb^*)a^*v\equiv a+ab \mod L_K(E)L_{ab},$$
 as we claimed. Consequently, for finite digraphs Proposition \ref{linear} is an obvious consequence of Corollary \ref{endoring}.

Let $v$ be an infinite emitter and $M$ be a $K$-space spanned by the basis
$F(E)v\cup \{a^*=a^*v\}\cup \{F(E)ba^*=F(E)vba^*v\}$. $M$ becomes a left $L_K(E)$-module by assigning to $u\in E^0$ and $x\in E^1$ linear transformations $\ta_u, \ta_x$ and $\ta_{x^*}$ defined as left multiplications by $u, x$, and $x^*$ on basis elements, i.e., by concatenation and erasure  according to the Cuntz-Krieger condition (CK1), respectively, except $\ta_a(a^*)=v;\, \ta_{a^*}a^*=a^*, \, \ta_{b^*}a^*=v$ and $\ta_{x^*}v=0$ for all $x\in s^{-1}(v)\setminus \{a\}$. It is a routine task to check that $\ta$ induces an algebra homomorphism from $L_K(E)$ into $\End(_KM)$ and therefore $\ann_{L_K(E)}v=L_K(E)L_{ab}$ whence the assignment $v\in M\mapsto v+L_K(E)L_{ab}\in T_{ab}$ defines an isomorphism between $_{L_K(E)}V$ and $T_{ab}$, and $v-a-ab\notin L_{ab}$. By convention we shall write $rm$ for $\ta_r(m)$ where $r$ or $m$ runs over elements from $L_K(E)$ or $M$, respectively.  Since $w=v-a-ab\in V$ is clearly annihilated by all $a\neq x (\ta_x)\in s^{-1}(v)$ and $a^*w=a^*-v-b$ is annihilated by all $x\in s^{-1}(v)$ the submodule $L_K(E)w=N$ of $_{L_K(E)}M$ has length 2 with composition factors isomorphic to the Cohn-Jacobson module $S_v$ defined by the infinite emitter $v$.

Therefore verifying Proposition \ref{linear} reduces to the simplicity of $S_{ab}=L_K(E)/L_K(E)(v-a-ab)$ and to compute its endomorphism ring. For the first aim, because $\{b^*a^*- v, (a^*)^2- a^*, x^*a^*\,|\, x\in s^{-1}(v)\setminus \{a, b\}\}\subseteq L$, every nonzero element of $S_{ab}$ can be represented
 by $z=\sum_i k_i\m_i +\sum_j k_j\n_ja^* \, (0\neq k_i, k_j \in K; \m_i \in F(E)v, \n_j\in \{v, F(E)b\}$ having the least number $N$ of nonzero coefficients, and that subject to this condition, the degree $\deg z$ of $z$, i.e., $\max_{i,j} \{|\m_i|, |\n_j|\}$ is minimal.  We shall use (double)
induction. The cases $N=1,\ 2$ are obvious. Assume $N=3$. Multiplying $z$ with the path $\g^*$ of minimal length, one can assume without loss of generality that one term of $z$ is either $v$ or $a^*$. Let $\m, \n\in vF(E)v$ be paths of $z$. Then one of them is a head of the other, say, $\m$ is a head of $\n$. Then $h_{\m}(1)\in s^{-1}(v)$. If $h_{\m}(1)\neq a$, then $h^*_{\m}(1)z$ is a sum of at most 2 terms and so the claim is verified. Thus $h_{\m}(1)=a$. Put $z_1=a^*z$. Then $z_1=a^*+\cdots$. Hence $b^*z_1=z_2=v+\cdots$ is a sum of 3 terms of a smaller degree. Therefore after finitely many steps one obtains that $z$ reduces to a sum $z=k_1v+k_2a+k_3ab\, (0\neq k_i\in K)$. Then $a^*z=k_1a^*+k_2a+k_3b$. Hence $(a^*)^2z=(k_1+k_2)a^*=0$ iff $k_2=-k_1$. In this case   
$b^*a^*z=(k_1+k_3)v=0$ iff $k_3=-k_1$. However, the case $k_2=k_2=-k_1$ cannot happen because $v-a-ab\in L$! Therefore the case $N=3$ is verified.  Assume now the claim holds for a number of non-zero coefficients of $z$ at most $N-1$ and all degrees and consider the case when the number of non-zero coefficients of $z$ is $N(\geq 4)$. The minimality of $N$ implies that no proper subsums of $z$ belong to $L$.  Multiplying $z$ with the path $\g^*$ of minimal length, one can assume again without loss of generality that one term of $z$ is either $v$ or $a^*$ and a path in $z$ of smaller length must be a head of longer ones, otherwise $\a^*z\notin L $ would have fewer terms whence the claim. However, for a proper path $\a$ of $z$ of the smallest length, $\a^*z$ must involve a subsum $k_1v+k_2a^*$ otherwise $\a^*z$ could have fewer non-zero coefficients, whence the claim by induction. Consequently, $b^*\a^*z\notin L$ have fewer non-zero coefficients and so our induction is complete and $S_{ab}$ is simple.
To compute the endomorphism ring of $S_{ab}$ one observes that an endomorphism of $S_{ab}$ is completely determined by an image $z$ of $v+L$. By multiplying with appropriate $r\in KE^*$ one gets $zr\in L_{ab}$, that is, every endomorphism of $S_{ab}$ is induced by an endomorphism of $KE^*/L_{ab}$ . This shows $\End(S_{ab})\cong K$, as stated.
\end{proof}
The similar argument and calculation to Proposition \ref{linear} shows that for a vertex $v\in E^0$ having two loops $a, b$ and $f=v-a^2-a-ab$ the left ideal $L_K(E)f$ contains the left ideal $L_{ab}$
of $KE^*$ generated by $\{f^*_1=(a^*)^2-a^*-v, f^*_2=b^*a^*-v; u\in E^0\setminus v\}$ and all ghost paths from a semigroup generated by $a^*$ and $b^*a^*$. Moreover, if $v$ is a regular vertex, then $L_K(E)f=L_K(E)L_{ab}$ holds.Therefore if $K=\Q$ the field of rationals, then $f^*_1$ is irreducible over $\Q$ and $\Q \langle a^*, b^*a^*\rangle$ maps to $\Q(\sqrt{5})$ by sending $b^*a^*\mapsto 0$ and $a^*\mapsto \frac {1+\sqrt{5}}{2}$. Completely similar to Proposition \ref{linear} we have
\begin{corollary}\label{nonlinear} If a vertex $v\in E^0$ has two loops $a, b$ and $f=v-a^2-a-ab$, then $L_K(E)f$ contains the left ideal $L_{ab}$
of $KE^*$ generated by $\{f^*_1=(a^*)^2-a^*-v, f^*_2=b^*a^*-v; u\in E^0\setminus v\}$ and all ghost paths from a semigroup generated by $a^*$ and $b^*a^*$. Moreover, $L_K(E)f=L_K(E)L_{ab}$ holds if $v$ is a regular vertex. In particular, $T_{ab}=L_{\Q}(E)/L_{\Q}(E)f$ is an irreducible representation of $L_{\Q}(E)$ whose endomorphism ring isomorphic to $\Q (\sqrt{5})$ if $v$ is a regular vertex and $\Q$ is the field of rationals. If $v$ is an infinite emitter and $\bar f$ is the image of $f$ in $T_{a}$, then
$0\subseteq L_{\Q}(E)a^*{\bar f}\subseteq L_{\Q}(E){\bar f}\subseteq T_{ab}$; two of simple subfactors are isomorphic to the Cohn-Jacobson module and the third one is  $T_{ab}/L_{\Q}(E)f$ having the endomorphism ring isomorphic to $\Q(\sqrt{5})$.
\end{corollary}
We mention the above consequence because $f^*_1$ is not a linear irreducible polynomial!
The examples discussed above suggest a practical, useful way to construct particular modules possibly of finite length over $L_K(E)$. Namely, one can take an arbitrary element $f\in KE$. One can assume without loss of generality that all paths in $f$ start from a vertex $v$ and $f$ contains non-zero constant term $kv$. Then one can study
first
$KE^*v/\ann_{KE^*v}f$ and then define a representation space for $L_K(E)/L_K(E)f$ and compute its endomorphism ring.
\vskip 0.3cm
{\bf Acknowledgements}
\vskip 0.25cm

  
The author is very grateful to M. Siddoway for his generous assistance in the writing of this article. His editing support makes this work more enjoyable for the audience. 
  
\bibliographystyle{amsplain}

\end{document}